\documentclass[12pt]{article}
\usepackage{amsthm}
\usepackage{varioref}
\usepackage{amsmath,amstext,amsthm,a4,amssymb,amscd}   
\usepackage[T1]{fontenc}
\usepackage[utf8]{inputenc}
\usepackage{lmodern}
\usepackage{array}
\usepackage{latexsym}
\usepackage{fancyhdr}
\usepackage{mathrsfs}\let\cal\mathscr
\usepackage[all]{xy}
\usepackage{makeidx}   
\usepackage{color}
\usepackage{pifont}
\usepackage{amsfonts,amssymb}
\usepackage{hyperref}
\usepackage[numbers,square]{natbib}
\usepackage{verbatim}
\usepackage{relsize}

\usepackage{avant}
\usepackage[T1]{fontenc}      

\usepackage{amsfonts,amssymb}  
 
\usepackage{graphicx}

\usepackage{natbib}
\newcommand \Om {\Omega}
\newcommand \om {\omega}
\newcommand \0 {\emptyset}

\renewcommand \leq {\leqslant}
\renewcommand \geq {\geqslant}
\newcommand{\longhookrightarrow}{\lhook\joinrel\longrightarrow}

\DeclareMathOperator{\Vol}{Vol}
\DeclareMathOperator{\End}{End}

\DeclareMathOperator{\Tr}{Tr}

\DeclareMathOperator{\Ker}{Ker}
\DeclareMathOperator{\rk}{rk}
\DeclareMathOperator{\GL}{GL}
\DeclareMathOperator{\Spec}{Spec}

\DeclareMathOperator{\Span}{Span}

\DeclareMathOperator{\Ric}{Ric}

\DeclareMathOperator{\Aut}{Aut}
\DeclareMathOperator{\Herm}{Herm}
\DeclareMathOperator{\scal}{scal}
\DeclareMathOperator{\Prod}{Prod}
\DeclareMathOperator{\Met}{Met}
\DeclareMathOperator{\Hilb}{Hilb}
\DeclareMathOperator{\Kod}{Kod}
\DeclareMathOperator{\FS}{FS}
\DeclareMathOperator{\ev}{ev}

\newcommand \dbar {\overline{\partial}}

\newcommand \< {\mathcal{h}}
\renewcommand \> {\mathcal{i}}

\newcommand \cinf {\CC^\infty}

\newcommand \Id {{\rm Id}}

\renewcommand \epsilon {\varepsilon}

\newcommand \CC {{\cal C}}
\newcommand \BB {{\cal B}}

\newcommand \HH {{\cal H}}

\newcommand \TT {{\cal T}}

\def\cL{\mathscr{L}}

\def\cB{\mathcal{B}}

\def\cL{\mathscr{L}}

\def\cT{\mathscr{T}}

\newcommand{\til}[1]{\widetilde{#1}}

\newcommand \dt {\frac{\partial}{\partial t}}

\newcommand \D[1] {\frac{\partial}{\partial #1}}

\newcommand \R {\mathbb R}
\newcommand \IP {\mathbb P}

\newcommand \CP {\mathbb C\mathbb P}
\newcommand \C {\mathbb C}

\newcommand \N {\mathbb N}

\newcommand \fl {\rightarrow}

\newcommand \ignore[1] {}

\theoremstyle{plain}
\newtheorem{theorem}{Theorem}[section]
\newtheorem{lem}[theorem]{Lemma}
\newtheorem{cor}[theorem]{Corollary}
\newtheorem{prop}[theorem]{Proposition}
\theoremstyle{definition}
\newtheorem{defi}[theorem]{Definition}
\newtheorem{ex}[theorem]{Example}

\newtheorem{rmk}[theorem]{Remark}

\numberwithin{equation}{section}
\hypersetup{
    colorlinks=true,
    linkcolor=red,
    filecolor=magenta,      
    urlcolor=cyan,
}

\begin{document}

\title{\bf{Quantization of symplectic fibrations and canonical metrics}}
\author{Louis IOOS$^1$ and Leonid POLTEROVICH$^2$}
\date{}
\maketitle
\newcommand{\Addresses}{{
  \bigskip
  \footnotesize

  \textsc{Université de Cergy-Pontoise, 95000 Cergy,
France}\par\nopagebreak
  \textit{E-mail address}: \texttt{louis.ioos@cyu.fr}
  
  	\medskip
	
	  \textsc{School of Mathematical Sciences, Tel Aviv University, Ramat Aviv, Tel Aviv 69978,
Israel}\par\nopagebreak
  \textit{E-mail address}: \texttt{polterov@tauex.tau.ac.il}

}}

\footnotetext[1]{Partially supported by the European Research Council Starting grant 757585}
\footnotetext[2]{Partially supported by the Israel Science Foundation
grant 1102/20}

\begin{abstract}
We relate Berezin-Toeplitz quantization of higher rank vector bundles
to quantum-classical hybrid systems and quantization in
stages of symplectic fibrations. We apply this picture to the analysis
and geometry of vector bundles, including the spectral gap of the Berezin
transform and the convergence rate of Donaldson's iterations towards
balanced metrics on stable vector bundles. We also establish
refined estimates in the scalar case to compute the
rate of Donaldson's iterations towards
balanced metrics on Kähler manifolds with
constant scalar curvature.


\end{abstract}

\section{Introduction}

The goal of a \emph{quantization process} is
to associate a system of quantum mechanics with any given
system of classical mechanics, in such a way that one recovers the
laws of classical mechanics from the laws of quantum mechanics at
the limit when the \emph{Planck constant} $\hbar$ tends to $0$.
However, in many concrete physical situations, one also
encounters \emph{quantum-classical hybrid systems} \cite{Elz12}, where some
degrees of freedom remain quantized while the others are taken
to be purely classical. The archetypical example is the celebrated
\emph{Stern-Gerlach experiment}, which describes in first approximation
a classical particle with quantized spin exhibiting a highly
non-classical behavior. From the point of view of quantization,
such quantum-classical hybrids are naturally described as the intermediate
step in
the process of \emph{quantization in stages} \cite[\S\,4.1]{GLS96}, where
instead of quantizing all degrees of freedom at once, one first
quantizes some specific degrees of freedom, for instance the internal
degrees of freedom of a classical particle, and then
quantizes the remaining classical degrees of freedom of 
a quantum-classical hybrid.

In this paper, we consider the process of quantization in
stages in the context of \emph{Berezin-Toeplitz quantization}
\cite{Ber74,BMS94} of a compact \emph{quantizable} symplectic manifold
$(X,\om)$, so that the cohomology class $[\om]\in H^2(X,\R)$ is integral.
Assuming the existence of an integrable complex structure $J\in\End(TX)$
compatible
with $\om$, this implies the existence of a holomorphic Hermitian line
bundle $(L,h^L)$ \emph{prequantizing} $(X,\om)$, so that its
\emph{Chern curvature} $R^L\in\Om^2(X,\C)$
satisfies
\begin{equation}\label{preq}
\om=\frac{\sqrt{-1}}{2\pi} R^L\,.
\end{equation}
Given a smooth volume form $d\nu_X$ on $X$, one can then define
the associated \emph{Hilbert space of quantum states}
\begin{equation}\label{quantspace}
\HH:=\left(H^0(X,L),\<\cdot,\cdot\>\right)
\end{equation}
as the finite-dimensional space $H^0(X,L)$ of
holomorphic sections of $L$ endowed with the $L_2$-Hermitian product
$\<\cdot,\cdot\>$ defined for any $s_1,\,s_2\in\HH$ by
\begin{equation}\label{L2}
\<s_1,s_2\>:=\int_X\,h^L(s_1(x),s_2(x))\,d\nu_X(x)\,.
\end{equation}
In the context of quantization, 
the space of smooth functions
$\cinf(X,\R)$ over $X$ represents classical observables, describing
classical mechanics over the phase space $(X,\om)$, while
the space of Hermitian operators $\Herm(\HH)$ acting on $\HH$ represents
quantum observables, describing quantum mechanics over the Hilbert
space of quantum states $\HH$.
The \emph{Berezin-Toeplitz quantization} of $(X,p\,\om)$ is
the map $T:\cinf(X,\R)\to\Herm(\HH)$ defined for any $f\in\cinf(X,\R)$
and $s\in\HH$ by
\begin{equation}\label{BTdefintro}
T(f)=P\,fs\,,
\end{equation}
where $fs\in\cinf(X,L)$ is the multiplication of $s\in\HH\subset\cinf(X,L)$
by $f$ and $P:\cinf(X,L)\to\HH$ denotes the orthogonal projection with respect
to \eqref{L2}.
In this context, one usually introduces a \emph{quantum number} $p\in\N$,
and considers the quantization of
the symplectic manifold $(X,p\,\om)$ for any $p\in\N$, which is
prequantized by the $p$-th tensor power $L^p:=L^{\otimes p}$ endowed with
the induced Hermitian metric. The integer 
$p\in\N$ can be thought as inversely proportional to the Planck
constant $\hbar$, and the
limiting regime when $p$ tends to infinity
is called the \emph{semi-classical limit}.

Let us now consider a \emph{symplectic fibration} $\pi:(M,\om^M)\to (X,\om)$,
where $(X,\om)$ is a compact quantizable symplectic manifold and $M$
is a compact manifold endowed with a closed $2$-form $\om^M$
which restricts to a quantizable symplectic form in each fibre.
Assume now that both manifolds admit integrable complex structures
compatible with all the data, let $\nu_\pi$ be a volume form in the fibres
of $\pi:M\to X$ depending smoothly on the base and
let $(\cL,h^{\cL})$ be a holomorphic Hermitian
line bundle over $M$ prequantizing the fibres of $\pi:(M,\om^M)\to (X,\om)$.
Up to replacing $\om^M$ by $r\,\om^M$
and $\cL$ by $\cL^r$ for $r\in\N$ large enough,
we can also assume that the dimension of the space of quantum states associated
with $(\pi^{-1}(x),\om^M|_{\pi^{-1}(x)})$ does not depend on
$x\in X$.
We then get an induced holomorphic Hermitian vector bundle
$(E,h^E)$ over
$(X,\om)$,
called the \emph{quantum-classical hybrid} associated with
$\pi:(M,\om^M)\to (X,\om)$, whose fibre over any $x\in X$ is
given by
\begin{equation}\label{E_x}
E_x:=H^0(\pi^{-1}(x),\cL|_{\pi^{-1}(x)})\,,
\end{equation}
and whose Hermitian metric $h^E$ at $x\in X$ is given
for any $s_1,\,s_2\in E_x$ by
\begin{equation}\label{L2hE}
h^E_x(s_1,s_2):=\int_{\pi^{-1}(x)}\,h^{\cL}(s_1(y),s_2(y))
\,d\nu_{\pi}(y)\,.
\end{equation}
The space $\cinf(X,\Herm(E))$ of Hermitian endomorphisms of $E$ over
$X$ then represents
the space of quantum observables in the fibres which remain classical
along the base, and we get a natural Berezin-Toeplitz quantization map
$T_{\pi}:\cinf(M,\R)\to\cinf(X,\Herm(E))$ in the fibres.

As explained for instance in \cite[\S\,4.1]{GLS96}, the construction
of the quantum-classical hybrid associated with a symplectic fibration
is the first step of a two-step process called \emph{quantization in
stages}. As a second step, we introduce a quantum number $p\in\N$,
a holomorphic Hermitian line bundle $(L,h^L)$ prequantizing $(X,\om)$
and a smooth volume form $d\nu_X$ over $X$, then consider the Hilbert space
$\HH_p$ of quantum states to be the
space $H^0(X,E_p)$ of holomorphic sections of the tensor product
\begin{equation}\label{Ep}
E_p:=E\otimes L^p\,,
\end{equation}
endowed with the $L_2$-Hermitian product \eqref{L2} using instead the
Hermitian metric $h^{E_p}$ on $E_p$ induced by $h^L$ and $h^E$.
Then by definition \eqref{E_x} of $E$, we have a natural identification
\begin{equation}\label{Hp=HEp}
H^0(X,E_p)\simeq H^0(M,\cL\otimes\pi^*L^p)\,,
\end{equation}
preserving the respective natural $L_2$-Hermitian products
as in formula \eqref{L2} for the smooth volume form $d\nu_M$ on $M$ defined
by 
\begin{equation}\label{nupidef}
d\nu_M:=d\nu_\pi\pi^*\,d\nu_X\,.
\end{equation}
Via formula \eqref{preq} applied to $\cL\otimes\pi^*L^p$,
the identification \eqref{Hp=HEp} thus states that
the Hilbert space $\HH_p$ of quantum states associated with \eqref{Ep}
as above naturally coincides with the space of quantum states
associated with the symplectic manifold $(M,\om^M+p\,\om)$, for $p\in\N$
big enough. 

In Section \ref{BTsec}, we introduce the general setting of Berezin-Toeplitz
quantization as a theory of quantum measurements, following \cite{IKPS19},
then describe the process of quantization in stages in this context.
We then study the basic properties of the Berezin-Toeplitz quantization map
for vector bundles $T_{E_p}:\cinf(X,\Herm(E))\to\Herm(\HH_p)$
introduced in \cite{MM08b},
defined as before with $\HH_p$ the
common Hilbert space \eqref{Hp=HEp} for any $p\in\N$.

In Section \ref{specsec}, we
extend to the setting of quantization in
stages our study in \cite{IKPS19} of
Berezin-Toeplitz quantization as a theory of quantum measurements.
In particular, we introduce
and study in Section \ref{specbersec} a natural notion of a
\emph{Berezin transform} $\BB_{E_p}$
associated with the Berezin-Toeplitz quantization of
a holomorphic Hermitian vector bundle
$(E,h^E)$ over $(X,p\,\om)$.
This Berezin transform is a bounded operator with discrete
spectrum acting on $\cinf(X,\Herm(E))$,
and is introduced in Definition \ref{BpdefE}
as the result of quantization followed by dequantization.
Its spectrum thus describes the \emph{quantum noise}
introduced by Berezin-Toeplitz quantization, which can be interpreted
as a global version of the \emph{Heisenberg uncertainty principle}.

To describe our main result in this respect,
consider the Riemannian metric on $X$ defined by
\begin{equation}\label{gTX}
g^{TX}:=\om(\cdot,J\cdot)\,,
\end{equation}
and recall that the induced Riemannian measure coincides with
the \emph{Liouville measure}
of the symplectic manifold $(X,\om)$, defined by
\begin{equation}\label{voldef}
d\nu:=\frac{\om^n}{n!}\,.
\end{equation}
For any holomorphic vector bundle
$(E,h^E)$,  we write
$\dbar$ for the
holomorphic $\dbar$-operator of the holomorphic vector
bundle $\End(E)$ of endomorphisms of $E$.
In that case, there is a natural extension of the Laplace-Beltrami
operator $\Delta$ of $(X,g^{TX})$, given by twice the
\emph{Kodaira Laplacian}, which acts
on the smooth sections $\cinf(X,\End(E))$
by the formula
\begin{equation}\label{KodLapdef}
\Box:=2\,\dbar^*\,\dbar\,,
\end{equation}
where $\dbar^{*}:\Om^1(X,\End(E))\to\cinf(X,\End(E))$
denotes the formal adjoint of
$\dbar:\cinf(X,\End(E))\to\Om^1(X,\End(E))$ with respect to the 
natural $L_2$-products on both spaces induced by $h^E$ and $g^{TX}$.
It has discrete positive spectrum, and we write
\begin{equation}\label{lapev}
0=\lambda_0^E\leq\lambda_1^E\leq\cdots\leq\lambda_k^E\leq\cdots
\end{equation}
for the increasing sequence of its eigenvalues counted with multiplicity.
On the other hand, let us write
\begin{equation}\label{berev}
1=\gamma_{0,p}^E\geq\gamma_{1,p}^E\geq\cdots\geq\gamma_{k,p}^E
\geq\cdots\geq 0
\end{equation}
for the decreasing sequence of the eigenvalues of the Berezin transform
$\cB_{E_p}$, counted with multiplicities.
In the special case $E=\C$, the Berezin transform $\BB_{E_p}$
reduces to the usual Berezin transform considered in \cite{IKPS19},
and the Kodaira Laplacian reduces to the usual Laplace-Beltrami
operator of $(X,g^{TX})$.
In this case, we simply remove the superscript $E$
in the notations \eqref{lapev} and \eqref{berev} for their eigenvalues.
In Section \ref{BTKEsec}, we give the proof of the following result.

\begin{theorem}\label{mainthvb}
For every integer $k\in\N$, we have the following
asymptotic estimate as $p\to +\infty$,
\begin{equation}\label{eq-LB}
\gamma_{k,p}^E =1-\frac{1}{4\pi p}\lambda_k^E + O(p^{-2})\;.
\end{equation}

In the special case $E=\C$ with
$g^{TX}$ \emph{Kähler-Einstein} and
$\nu_X$ the Liouville measure,
we have the following refined estimate as $p\fl+\infty$,
\begin{equation}\label{mainthKE}
\gamma_{k,p} =1-\frac{\lambda_k}{4\pi p}+
\frac{\lambda_k^2+2c\lambda_k}{32\pi^2p^2}+
O(p^{-3})\;.
\end{equation}
\end{theorem}

This shows in particular that the Berezin transform admits a positive
\emph{spectral gap}, controlled by the first positive
eigenvalue of the Kodaira Laplacian
at the semiclassical limit as $p\to +\infty$.
In \cite[Th.\,3.1]{IKPS19},
the estimate \eqref{eq-LB} was established in the case $E=\C$.
Theorem \ref{mainthvb} thus extends this result in two directions,
to the case of a general vector
bundle $E$ on one hand, and gives the refined estimate \eqref{mainthKE}
in the case $E=\C$ but with the additional
assumption that the metric Kähler-Einstein on the other.
The proof makes use of the quantum-classical correspondance for the
Berezin-Toeplitz
quantization of vector bundles established by Ma and Marinescu
in \cite{MM08b,MM12}.


In Section \ref{balsec},
we complete the study started in \cite[\S\,4]{IKPS19} and \cite{Ioo20,Ioo21}
of the applications of Berezin-Toeplitz quantization to
\emph{Donaldson's program}
towards the search for canonical K\"ahler metrics in complex geometry.
for the approximation of Kähler metrics of \emph{constant scalar curvature}.
In fact, Donaldson introduced
in \cite{Don01,Don05} a finite dimensional approximation of Kähler metrics of
\emph{constant scalar curvature}, given by a sequence of Hermitian inner
products on $H^0(X,L^p)$ for all $p\in\N$,
called \emph{balanced products}. Writing $\Prod(H^0(X,L^p))$
for the symmetric space of Hermitian inner products on $H^0(X,L^p)$,
balanced products are defined as
the fixed points of the dynamical system $\TT_p$ acting
on $q\in\Prod(H^0(X,L^p))$ by the formula
\begin{equation}\label{Tdefcan}
\TT_p(q)=\frac{\dim\,H^0(X,L^p)}{\Vol(X,\nu_q)}\,\int_{X}
h^{FS}_q(\cdot,\cdot)\,d\nu_q(x)\,,
\end{equation}
where $h^{FS}_q$ denotes the pullback of the
Fubini-Study metric induced
by $q$ with respect to the canonical \emph{Kodaira embedding}
$X\hookrightarrow\IP(H^0(X,L^p)^*)$ and $\nu_q$
denotes the Liouville measure associated with the symplectic form $\om_q$ defined
by \eqref{preq}.
In this context, Donaldson shows in \cite{Don01} that these balanced
metrics converge as $p\to+\infty$ towards the unique K\"ahler
metric of constant scalar curvature, when it exists and when the
holomorphic automorphism group
$\Aut(X,L)$ of $(X,J,L)$ is discrete. The convergence of
the iterations of $\TT_p$ has been established by Donaldson in \cite{Don05}
and Sano in \cite[Th.\,1.2]{San06}. In Section \ref{scalspecsec},
we establish and compute the exponential speed of convergence of these
iterations via a refined version of Theorem \ref{mainthvb}
using instead the eigenvalues
of the operator $D$ acting on $f\in\cinf(X,\R)$ by the formula
\begin{equation}\label{varscalopdef}
D f:=\dt\Big|_{t=0}\scal\left(\om+t\frac{\sqrt{-1}}{2\pi}
\dbar\partial f\right)\,,
\end{equation}
where $\om$ is the K\"ahler form associated with the constant scalar curvature
metric and the map $\scal$
sends a Kähler form to the scalar curvature of the associated Riemannian
metric \eqref{gTX}. As explained \cite[Def.\,4.3,\,Lemma\,4.4]{Sze14},
this operator is an elliptic self-adjoint
differential operator of fourth order called the \emph{Lichnerowicz operator},
and as such, it admits a discrete spectrum.

In Section \ref{vbcase}, we study the extension of
this program due to Wang in \cite{Wan05}
on the approximation of
\emph{Hermite-Einstein metrics} on a
holomorphic vector bundle $E$ over $(X,\om)$, which are the
Hermitian metrics $h^E$ on $E$ whose the Chern curvature
$R^E\in\Om^2(X,\End(E))$ satisfies
\begin{equation}\label{weakHE}
\frac{\sqrt{-1}}{2\pi}\<\om,R^E\>_{g^{TX}}=c\,\Id_E
\end{equation}
for some constant $c\in\R$, where $\<\cdot,\cdot\>_{g^{TX}}$
denotes the pairing on $\Om^2(X,\End(E))$ with values in $\End(E)$
induced by $g^{TX}$.
By a fundamental result of Uhlenbeck-Yau in \cite{UY86} and
Donaldson in \cite{Don87}, a holomorphic vector bundle $E$ is
\emph{Mumford stable}
if and only if it admits a Hermite-Einstein metric $h^E$
and if it is \emph{simple}, meaning that it cannot be decomposed
as a direct sum of holomorphic vector bundles of smaller dimension.
In this context, we fix a measure $\nu$ over $X$
and define a dynamical system $\cT_{E_p}$
acting on a Hermitian inner product
$q\in\Prod(H^0(X,E_p))$ by the formula
\begin{equation}\label{DonitflaEintro}
\cT_{E_p}(q):=\frac{\dim H^0(X,L^p)}{\Vol(X,\nu)\rk(E)}\int_X\,h^{FS}_q(\cdot,\cdot)\,d\nu(x)\,,
\end{equation}
where $h^{FS}_q$ is the pullback to $E$ of the Fubini-Study metric
induced by $q$ via the canonical
\emph{Grassmanian embedding}
$X\hookrightarrow\mathbb{G}(\rk(E),H^0(X,E_p)^*)$.
The Hermitian metric $h^{FS}_q$ is then called \emph{$\nu$-balanced}
if $q$ is a fixed point of \eqref{DonitflaEintro}.
Assuming that $\nu$ is the Liouville measure,
Wang showed in \cite{Wan05}
that $\nu$-balanced metrics on $E$ converge as $p\to+\infty$ to the
Hermite-Einstein metric up to an explicit conformal factor,
and Seyyedali established the convergence
of the iterations of \eqref{DonitflaEintro} in \cite{Sey09}.
Let us point out that the fact that the volume form appearing in \eqref{DonitflaEintro} is fixed makes the study of its fixed points
a much simpler matter than with the original dynamical system \eqref{Tdefcan}
due to Donaldson in \cite{Don05}.
In Section \ref{vbcase}, we prove the following result.

\begin{theorem}\label{gapvb}
1) Assume that $E$ is
Mumford stable and that $\nu$ is the Liouville measure.
Then there exists
$p_0\in\N$ such that for any
$p\geq p_0$ and $q\in\Prod(H^0(X,E_p))$,
there exists a $\nu$-balanced product $q_p\in\Prod(H^0(X,E_p))$
and constants $C_p>0,\,\beta_p\in(0,1)$ such that for all $r\in\N$,
we have
\begin{equation}\label{expcvest}
\textup{dist}\left(\TT_{E_p}^r(q)\,,\,q_p\right)\leq
C_p\beta^r_p\;,
\end{equation}
and such that as $p\fl+\infty$,
\begin{equation}\label{betap}
\beta_p=1-\frac{\lambda_1^E}
{4\pi p}+ o(p^{-1})\,,
\end{equation}
where $\lambda_1^E>0$ is the first positive eigenvalue of the
Kodaira Laplacian \eqref{KodLapdef} associated with
the Hermitian metric $e^f\,h^E$, where $h^E$
satisfies the Hermite-Einstein equation \eqref{weakHE}
and $f\in\cinf(X,\R)$ satisfies
$\Delta\,f=2\pi\big(\scal(\om)-\int_X\scal(\om)\,\frac{d\nu}{\Vol(X)}\big)$.

2) Assume now that $E=\C$, that $g^{TX}$ has constant scalar curvature
and that $\Aut(X,L)$ is discrete.
Then \eqref{expcvest} holds for the dynamical system \eqref{Tdefcan},
with constant $\beta_p\in(0,1)$ satisfying the following
estimate as $p\fl+\infty$,
\begin{equation}\label{betapbal}
\beta_p=1-\frac{\mu_1}{8\pi p^2}+ o(p^{-2})\,,
\end{equation}
where $\mu_1>0$ is the first positive eigenvalue of the operator
\eqref{varscalopdef}.
\end{theorem}

In \cite[Th.\,4.4, Rem.\,4.9]{IKPS19},
Part 1 of Theorem \ref{gapvb} was established in the case $E=\C$.
Theorem \ref{gapvb} thus extends this result in two directions,
to the case of a general vector
bundle $E$ on one hand, and to the case of the original notion of balanced
metrics instead of $\nu$-balanced metrics for $E=\C$ on the other.
While the rate of exponential convergence for $\nu$-balanced metrics
computed in \cite{IKPS19}
has been predicted by Donaldson in \cite{Don05}, there was in our knowledge
no such predictions concerning \eqref{betapbal} for the original
notion of balanced metrics. 

In Section \ref{momsec}, we establish remarkable identities
relating the dynamical systems \eqref{DonitflaEintro}
and \eqref{Tdefcan} to the \emph{moment maps} for balanced
embeddings introduced in \cite{Wan05} and \cite{Don01}.
In particular, we show in Remarks \ref{KMSrmk} and \ref{Finermk}
how Theorem \ref{gapvb} reduce to the
lower bounds of the differential of the balanced moment maps
established by Keller, Meyer and
Seyyedali in \cite{KMS16} and Fine in \cite{Fin12}. 
The methods presented
in this paper thus gives a natural interpretation of their lower bounds
as the quantum noise induced by Berezin-Toeplitz quantization.
Via the work of Phong and Sturm in \cite{PS04},
and as explained in \cite{Ioo20},
this gives a conceptual
explanation for the most crucial part of
the proof of the results of Donaldson in
\cite{Don01} and Wang in \cite{Wan05} on the existence of balanced metrics,
replacing a delicate
geometric argument by the use of the spectral gap of the Berezin transform.


Finally, in Section \ref{mathphi}, we describe
the Stern-Gerlach experiment as a fundamental example
of a quantum-classical hybrid, and we discuss
the
physical interpretation of the classical-quantum
correspondence for vector bundles established by Ma and Marinescu
in \cite[(0.3)]{MM12} via the
process of quantization in stages.

\medskip\noindent{\bf Acknowledgement.}
The authors wish to thank the anonymous referee for helpful suggestions,
which allowed great improvements in the presentation of this paper.

\section{Berezin-Toeplitz quantization in stages}
\label{BTsec}

In this Section, we introduce the two-step process of quantization in
stages in the context of Berezin-Toeplitz quantization, then study
the Berezin-Toeplitz quantization
of holomorphic Hermitian vector bundles as the quantization of
a quantum-classical hybrid.

In Section \ref{BTPOVMsec}, we describe the
Berezin-Toeplitz quantization of a compact symplectic manifold from the
point of view of quantum measurement theory. In Section \ref{stagesec},
we describe the process of quantization in stages in this context, and
in Section \ref{BTvbsec}, we describe the classical-quantum
correspondence for vector bundles established by Ma and Marinescu
in \cite[(0.3)]{MM12}.

\subsection{Berezin-Toeplitz quantization}
\label{BTPOVMsec}

Let $(X,\om)$ be a \emph{quantizable} compact symplectic manifold of dimension
$2n$, meaning that the cohomology class $[\om]\in\Om^2(X,\R)$
is integral. This is equivalent with the existence of a Hermitian line
bundle with Hermitian connection $(L,h,\nabla^L)$
\emph{prequantizing} $(X,\om)$, so that the curvature
$R^L\in\Om^2(X,\C)$ of $\nabla^L$ satisfies the \emph{prequantization
condition} \eqref{preq}.
We will always assume that $(X,\om)$ admits a compatible integrable
complex structure
$J\in\End(TX)$, making $(X,J,\om)$ into a \emph{Kähler manifold}
and $(L,h,\nabla^L)$ into a holomorphic Hermitian line
bundle equipped with its \emph{Chern connection}.
For any $p\in\N^*$, the symplectic manifold $(X,p\,\om)$
is again Kähler for the same complex structure, and is
prequantized by the
$p$-th tensor power $L^p:=L^{\otimes p}$ with induced
Hermitian metric $h^{L^p}$ and connection $\nabla^{L^p}$.

Let us fix a smooth volume form $d\nu_X$ on $X$, and recall that
the space of
\emph{quantum states} associated with $(X,p\,\om)$
is the finite dimensional Hilbert space $\HH_p$ of
holomorphic sections of
$L^p$ endowed with the $L_2$-Hermitian product $\<\cdot,\cdot\>_p$
defined as in \eqref{L2}.

From the point of view of quantum measurements,
a quantization process is described through
the following basic notion of \emph{Positive Operator Valued Measure},
abbreviated into POVM, where $\cB(X)$ denotes the $\sigma$-algebra
of Borel sets of $X$.

\begin{defi}\label{POVMdef}
A \emph{POVM}
acting on $\HH_p$ over $X$ is a $\sigma$-additive map
\begin{equation}\label{POVMfla}
W_p:\cB(X)\longrightarrow\Herm(\HH_p)
\end{equation}
with values in positive operators, satisfying
$W_p(\0)=0$ and $W_p(X)=\Id_{\HH_p}$.
\end{defi}

In the context of Berezin-Toeplitz
quantization, the fundamental tool used to relate
quantum states to classical ones is the
the following \emph{evaluation map} at $x\in X$,
\begin{equation}\label{ev}
\begin{split}
\ev_x:\HH_p&\longrightarrow L^p_x\\
s~&\longmapsto s(x)\,.
\end{split}
\end{equation}
We write $\ev^*_{x}:L^p_x\fl\HH_p$ for its dual
with respect to the Hermitian metric $h^{L^p}$ on $L^p_x$
and to the $L_2$-Hermitian product $\<\cdot,\cdot\>_p$ on $\HH_p$.

\begin{prop}\label{BTPOVMdef}
The $\Herm(\HH_p)$-operator valued measure defined for all $x\in X$ by
\begin{equation}\label{BTdensity}
dW_p(x):=\ev_x^*\ev_x\,d\nu_X(x)
\end{equation}
induces by integration a map $W_p:\cB(X)\longrightarrow\Herm(\HH_p)$
satisfying the axioms of Definition \ref{POVMdef}.
\end{prop}
\begin{proof}
The operator valued measure \eqref{BTdensity} induces by
integration a $\sigma$-additive map
$W_p:\cB(X)\longrightarrow\Herm(\HH_p)$ satisfying $W_p(\0)=0$.
Furthermore, for any $U\in\cB(X)$ and $s_1,\,s_2\in\HH_p$, formula
\eqref{ev} implies
\begin{equation}\label{Wp=1}
\begin{split}
\<W_p(U)\,s_1,s_2\>_p:&=\int_U\,
\<\ev_x^*\ev_x\,s_1,s_2\>_p\,d\nu_X(x)\\
&=\int_U\,h^p(s_1(x),s_2(x))\,d\nu_X(x)\,.
\end{split}
\end{equation}
Taking $s_1=s_2$ shows that $W_p(U)$ is a positive operator, and
taking $U=X$ shows that $W_p(X)=\Id_{\HH_p}$.
\end{proof}

The POVM induced by formula \eqref{BTdensity} is
called the \emph{Berezin-Toeplitz POVM} of $(X,p\,\om)$.
By the classical \emph{Kodaira vanishing theorem},
for all $p\in\N$ big enough and all $x\in X$,
the evaluation map \eqref{ev} is surjective,
so that its dual is injective. Thus there exists
a unique rank-$1$ projector $\Pi_p(x)\in\Herm(\HH_p)$, called
the \emph{coherent state projector} at $x\in X$, and
a unique positive function $\rho_p\in\cinf(X,\R)$, called the
\emph{Rawnsley function}, such that \eqref{BTdensity} takes the form
\begin{equation}\label{BTdensity2}
dW_p(x)=\Pi_p(x)\,\rho_p(x)\,d\nu_X(x)\,.
\end{equation}
The following Proposition
shows that we recover the definition
\cite[\S\,3.1,\,Rmk.\,3.12]{IKPS19} of a (weighted)
Berezin-Toeplitz POVM.

\begin{prop}\label{cohstateprojprop}
For any $x\in X$, the coherent state projector $\Pi_p(x)\in\Herm(\HH_p)$
is the unique orthogonal projector satisfying
\begin{equation}\label{cohstateprojfla}
\Ker\Pi_p(x)=\{\,s\in \HH_p~|~s(x)=0\,\}\,,
\end{equation}
and the Rawnsley function at $x\in X$ satisfies
\begin{equation}\label{Rp=evev}
\rho_p(x)=\ev_x\ev_x^*\in\Herm(L^p)\simeq\R\,.
\end{equation}
\end{prop}
\begin{proof}
By definition, the coherent state projector at $x\in X$ is
the orthogonal projection on the image of $\ev^*_{x}:L^p_x\fl\HH_p$.
Now by definition of the evaluation map \eqref{ev}, an element
$s\in\HH_p$ is orthogonal to the image of
$\ev^*_{x}:L^p_x\fl\HH_p$ if and only if for all $v\in L^p_x$,
we have
\begin{equation}
0=\<\ev^*_{x}.v,s\>_p=h^p(v,s(x))\,,
\end{equation}
that is, if and only if $s(x)=0$. As orthogonal projectors are
characterized by their kernels, this completes the proof of formula
\eqref{cohstateprojfla}.

To establish the identity \eqref{Rp=evev}, note that
\eqref{cohstateprojfla} implies that $\ev_x\Pi_p(x)=\ev_x$,
and \eqref{BTdensity2} that
$\rho_p(x)\Pi_p(x)=\ev^*_x\ev_x$. Thus for any $s\in\HH_p$, we get
\begin{equation}
\begin{split}
\left(\ev_x\ev_x^*\right)s(x)&=\ev_x\ev_x^*\ev_x s=\rho_p(x)\ev_x\Pi_p(x)s=\rho_p(x)s(x)\,,
\end{split}
\end{equation}
which proves formula \eqref{Rp=evev}.
\end{proof}

A POVM with a density of the form \eqref{BTdensity2} is called
a \emph{rank-$1$ POVM}, and the associated coherent state projector
$\Pi_p(x)\in\Herm(\HH_p)$ at $x\in X$
is then interpreted as the quantization of a classical particle at $x\in X$,
where quantum states are seen as
positive rank-$1$ operators acting on $\HH_p$.
This induces the natural notion of a \emph{coherent state quantization},
from which we recover the following Berezin-Toeplitz quantization
of classical
observables.

\begin{defi}\label{BTquantdef}
The \emph{Berezin-Toeplitz quantization} of $(X,p\,\om)$
is the linear
map defined by
\begin{equation}\label{BTmapfla}
\begin{split}
T_p:\cinf(X,\R)&\longrightarrow\Herm(\HH_p)\\
f\,&\longmapsto\,\int_X f(x)\,dW_p(x)\,.
\end{split}
\end{equation}
\end{defi}

The argument of equation \eqref{Wp=1}
shows that Definition \ref{BTquantdef} coincides with the usual
definition \eqref{BTdefintro}.

\subsection{Quantization in stages}
\label{stagesec}

In this section, we introduce the concept of quantization in stages,
and show how it can be described by the
Berezin-Toeplitz quantization of vector bundles
introduced in \cite{MM08b}, which we
interpret as the quantization of quantum-classical hybrids.

\begin{defi}\label{quantfib}
Let $\pi: M\fl X$ be a submersion between compact manifolds.
A Hermitian line bundle with Hermitian connection
$(\cL,h^{\cL},\nabla^{\cL})$ over $M$ is said to
\emph{prequantize} the fibration if the $2$-form
$\om^M\in\Om^2(X,\R)$ defined by the formula
\begin{equation}\label{preqfib}
\om^M:=\frac{\sqrt{-1}}{2\pi} R^{\cL}\,,
\end{equation}
where $R^{\cL}$ is the curvature of $\nabla^{\cL}$, restricts
to a symplectic form on the fibres of $\pi: M\fl X$.
\end{defi}

Let $\pi:(M,\om^M)\fl(X,\om)$ be a prequantized fibration with base
a compact prequantized symplectic manifold $(X,\om)$, and
assume that both $X$ and $M$
admit integrable complex structures compatible with $\om$ and
the restriction of $\om_M$ to the fibres, and
making $\pi: M\fl X$ holomorphic.
This endows $\pi:(M,\om^M)\fl(X,\om)$ with the structure of
a \emph{Kähler fibration} in the sense \cite[Def.\,1.4]{BGS88}.
Furthermore, this makes
$(\cL,h^{\cL},\nabla^{\cL})$ into
a holomorphic Hermitian line bundle equipped with its Chern
connection.

From now on, we fix smooth volume forms $d\nu_M$ and $d\nu_X$ over
$M$ and $X$ respectively, and
assume that the higher cohomology groups of $\cL$ in the fibres
satisfy $H^j(\pi^{-1}(x),\cL|_{\pi^{-1}(x)})=\{0\}$,
for all $j>0$ and $x\in X$.
The Riemann-Roch-Hirzebruch formula then
implies that $\dim H^0(\pi^{-1}(x),\cL|_{\pi^{-1}(x)})$
does not depend on $x\in X$.
Thanks to the \emph{Kodaira vanishing theorem}, this can always
be achieved replacing $\om^M$ by $r\om^M$ for $r\in\N$ large enough
in Definition \ref{preqfib}, so that $\cL$ is replaced by $\cL^r$.
This assumption allows us to make the following
definition, which is at the basis of the concept of quantization in stages.

\begin{defi}\label{qchybdef}
The \emph{quantum-classical hybrid} associated with the prequantized
Kähler fibration $\pi:(M,\om^M)\fl(X,\om)$ is the holomorphic
vector bundle $E$ over $X$ whose fibre at any $x\in X$
is the space \eqref{E_x}
of holomorphic sections of $\cL$ over $\pi^{-1}(x)$,
endowed with the $L^2$-Hermitian product \eqref{L2hE}
induced by $h^\cL$ and the smooth volume form $d\nu_{\pi}$ defined
over any fibre of $\pi:M\to X$ via formula
\eqref{nupidef}.
\end{defi}

The holomorphic structure of $E$ is defined through its
holomorphic sections over any open set
$U\subset X$ as the
holomorphic sections of $\cL$ over $\pi^{-1}(U)$, so that $E$ coincides
as a sheaf with the direct image of $\cL$ by $\pi:M\to X$.

Following Definition \ref{quantfib},
let us now consider $p\in\N$ big enough so that the $2$-form
$\om^M+p\,\pi^*\om\in\Om^2(M,\R)$ is non-degenerate.
This endows $M$ with a symplectic structure,
prequantized by the line bundle $\cL\otimes\pi^*L^p$ with induced metric
and connection. Then the identification
\eqref{Hp=HEp} between the Hilbert space of quantum states $\HH_p$
associated with $(M,\om_M+p\,\pi^*\om)$ and
the space of holomorphic sections of
$E_p:=E\otimes L^p$ follows from Definition \ref{qchybdef},
since holomorphic sections of $E_p:=E\otimes L^p$
are precisely the holomorphic sections of the fibre
depending holomorphically from the base.
The $L_2$-Hermitian product \eqref{L2} then satisfies
\begin{equation}\label{L2E}
\<s_1,s_2\>_p=\int_X\,h^{E_p}(s_1(x),s_2(x))\,d\nu_X(x)\,,
\end{equation}
for any $s_1,\,s_2\in\HH_p$ seen as holomorphic sections of $E_p$,
where $h^{E_p}$ denotes the Hermitian product on $E_p$ induced by $h^L$ and
$h^E$.
Hence the quantization of $(M,\om^M+p\,\pi^*\om)$ can be obtained
as a two-step process, called \emph{quantization in stages},
by first considering the holomorphic Hermitian vector bundle $(E,h^E)$
induced by the space of quantum states of the fibres of
$\pi:(M,\om^M)\fl(X,\om)$,
and then taking the quantization of
the vector bundle $(E,h^{E})$ over $(X,p\,\om)$ to be
the space of holomorphic sections of $E_p$ as above.
Note that in the limiting regime when $p$ tends to infinity,
the horizontal form $\pi^*\om$ dominates,
and the situation is then essentially different from
the previous section, as the $2$-form $\pi^*\om$ is degenerate
along the fibres of $\pi:M\fl X$. This regime is also called the
\emph{weak coupling limit} in \cite[\S\,4.5]{GLS96},
referring to the fact that the
symplectic form of the fibres becomes comparatively small.


Let us now extend the identification \eqref{Hp=HEp} of
the spaces of quantum states to a natural identification
of the respective Berezin-Toeplitz quantizations.
Proceeding by analogies, let us consider for any $x\in X$
the evaluation map
\begin{equation}\label{evE}
\begin{split}
\ev_{E_x}:\HH_p&\longrightarrow E_{p,x}\\
s~&\longmapsto s(x)\,,
\end{split}
\end{equation}
and write $\ev_{E_x}^*:E_{p,x}\fl\HH_p$ for its dual
with respect to $h^{E_p}$ and \eqref{L2E}.

\begin{prop}\label{BTPOVMdefE}
The map $W_{E_p}:\cB(X)\longrightarrow\Herm(\HH_p)$
induced by the
$\Herm(\HH_p)$-operator valued measure defined for all $x\in X$ by
\begin{equation}\label{BTdensityE}
dW_{E_p}(x):=\ev_{E_x}^*\ev_{E_x}\,d\nu_X(x)\,,
\end{equation}
defines a POVM in the sense of Definition \ref{POVMdef}.
\end{prop}
\begin{proof}
The proof is stricly analogous to the proof of Proposition \ref{BTPOVMdef}.
\end{proof}

For any Hermitian vector bundle $(E,h^E)$,
we will write $\Herm(E)$ for the bundle of Hermitian endomorphisms of $E$
over $X$. We will freely use the natural identification
\begin{equation}
\Herm(E)\simeq\Herm(E_p)\,.
\end{equation}
The following definition generalizes Definition \ref{BTquantdef}.

\begin{defi}\label{BTquantdefE}
The \emph{Berezin-Toeplitz quantization} of $(E,h^{E})$
over $(X,p\,\om)$ is the map
\begin{equation}\label{BTmapflaE}
\begin{split}
T_{E_p}:\cinf(X,\Herm(E))&\longrightarrow\Herm(\HH_p)\\
F\,&\longmapsto\,\int_X \ev_{E_x}^*F(x)\ev_{E_x}\,d\nu_X(x)\,.
\end{split}
\end{equation}
\end{defi}


We now have the following basic functoriality result.

\begin{prop}\label{Tfunct}
For any $p\in\N$ big enough,
the Berezin-Toeplitz quantization map $T_{E_p}:\cinf(X,\Herm(E))\to
\Herm(\HH_p)$ satisfies the formula
\begin{equation}
T_p=T_{E_p}\circ T_{\pi}\,,
\end{equation}
where $T_{\pi}:\cinf(M,\R)\to\cinf(X,\Herm(E))$
is the Berezin-Toeplitz
quantization of the fibres of $\pi:(M,\om^M)\to (X,\om)$,
and where $T_p:\cinf(M,\R)\to\Herm(\HH_p)$
is the Berezin-Toeplitz quantization of $(M,\om^M+p\pi^*\om)$ for the
measure \eqref{nupidef}.
\end{prop}
\begin{proof}
Recall from Definition \ref{qchybdef} that for any $x\in X$, the fibre
$E_{p,x}$ is naturally identified with the space of holomorphic
sections of $\cL\otimes\pi^*L^p$ restricted to $\pi^{-1}(x)$.
Then for any $y\in\pi^{-1}(x)$, consider
the evaluation maps
\begin{equation}
\begin{split}
\ev^M_y&:\HH_p\longrightarrow (\cL\otimes\pi^*L^p)_y~~\text{and}\\
\ev^{\pi^{-1}(x)}_y&:E_{p,x}\longrightarrow(\cL\otimes\pi^*L^p)_y\,,
\end{split}
\end{equation}
defined by formula \eqref{ev} over $M$
and $\pi^{-1}(x)$ respectively.
Then by definition, for all $y\in\pi^{-1}(x)$ we get
\begin{equation}\label{evpievE=ev}
\ev^{\pi^{-1}(x)}_y\,\ev_{E_x}=\ev^M_y\,,
\end{equation}
and for any $f\in\cinf(M,\R)$,
the Berezin-Toeplitz quantization of $(M,\om^M+p\pi^\om)$
as in Definition \ref{BTquantdef} satisfies
\begin{equation}\label{functcomput}
\begin{split}
T_p(f)&=\int_M f(y)\,\left(\ev^M_y\right)^*\ev^M_y\,d\nu_M(x)\\
&=\int_{x\in X}\int_{y\in\pi^{-1}(x)}f(y)\,
\ev_{E_x}^*\left(\ev^{\pi^{-1}(x)}_y\right)^*
\ev^{\pi^{-1}(x)}_y\ev_{E_x}\,
d\nu_{\pi}(y)\,d\nu_X(x)\\
&=\int_{x\in X}\ev_{E_x}^*\left(\int_{y\in\pi^{-1}(x)}f(y)\,
\left(\ev^{\pi^{-1}(x)}_y\right)^*\ev^{\pi^{-1}(x)}_y\,
d\nu_{\pi}(y)\right)
\ev_{E_x}\,d\nu_X(x)\\
&=\int_{X}\,\ev_{E_x}^*\,T_{\pi}(f)\,\ev_{E_x}\,d\nu_X(x)
=T_{E_p}\left(T_{\pi}(f)\right)\,.
\end{split}
\end{equation}
This gives the result.
\end{proof}

%

The following Definition is set
by analogy with formula \eqref{Rp=evev}
in the second part of Proposition \ref{cohstateprojprop}.

\begin{defi}\label{Rawnsect}
The \emph{Rawnsley section} $\rho_{E_p}\in\cinf(X,\Herm(E))$ is defined for
any $x\in X$ by the formula
\begin{equation}\label{Rawnsectfla}
\rho_{E_p}(x)=\ev_{E_x} \ev_{E_x}^*\in\End(E_{x})\,.
\end{equation}
\end{defi}

On the other hand, consider the $L_2$-scalar product defined on
$F_1,\,F_2\in\cinf(X,\Herm(E))$ by the formula
\begin{equation}\label{L2alphaE}
\<F_1,F_2\>_{W_{E_p}}:=\int_X\,\Tr^{E}[F_1(x)\,
\rho_{E_p}(x)\,F_2(x)]\,d\nu(x)\,.
\end{equation}

\begin{defi}\label{BersymbE}
The \emph{Berezin symbol} is the map
\begin{equation}
T_{E_p}^*:\Herm(\HH_p)\longrightarrow\cinf(X,\Herm(E))\,,
\end{equation}
dual to the Berezin-Toeplitz quantization map of Definition
\ref{BTquantdefE} with respect to the
scalar product \eqref{L2alphaE}.
\end{defi}


The Berezin symbol satisfies the following
functoriality property in the setting of quantization in stages,
which can be seen as an appropriate analogue of
Proposition \ref{Tfunct}.

\begin{prop}\label{T*funct}
In the setting and notations of Proposition \ref{Tfunct},
the Berezin symbol map of Definition \ref{BersymbE}
satisfies the following formula,
for any $A\in\Herm(\HH_p)$ and $x\in M$,
\begin{equation}\label{T*functfla}
T_p^*(A)=\Tr^{E_p}\left[\left(T_{E_p}^*(A)\circ\pi\right)\,\Pi_{\pi}
\right]\,,
\end{equation}
where $\Pi_{\pi}:\cinf(M,\R)\fl\cinf(X,\Herm(E))$
is the coherent state projector of Proposition \ref{cohstateprojprop}
associated with the Berezin-Toeplitz
quantization of the fibres of $\pi:(M,\om^M)\fl(X,\om)$.
\end{prop}
\begin{proof}
Using Proposition \ref{Tfunct}, it suffices to show that the map
\begin{equation}
\begin{split}
\cinf(X,\Herm(E))&\longrightarrow\cinf(M,\R)\\
F~&\longmapsto\Tr^{E}[(F\circ\pi)\,\Pi_\pi]\,,
\end{split}
\end{equation}
is dual to $T_\pi$ with respect to the $L_2$-scalar product
\eqref{L2alphaE} associated with $(M,\om^M+p\,\pi^*\om)$ for
$E=\C$ and the $L^2$-scalar product \eqref{L2alphaE}
associated with $(E,h^E)$ over $(X,p\,\om)$.

Using formula \eqref{evpievE=ev} for $x\in X$ and $y\in\pi^{-1}(x)$,
we see that the Rawnsley function $\rho_p\in\cinf(M,\R)$
associated with the Berezin-Toeplitz quantization
of
$(M,\om^M+p\,\pi^*\om)$ satisfies
\begin{equation}
\begin{split}
\rho_p(y)&=\ev^{\pi^{-1}(x)}_y
\rho_{E_p}(x)\left(\ev^{\pi^{-1}(x)}_y\right)^*\\
&=\Tr^{E}\left[\left(\ev^{\pi^{-1}(x)}_y\right)^*
\ev^{\pi^{-1}(x)}_y\rho_{E_p}(x)\right]\,.
\end{split}
\end{equation}
Then for any $F\in\cinf(X,\Herm(E_p))$, recalling
that $\Pi_{\pi}(y)$ is a projection on the $1$-dimensional
image of $\ev^{\pi^{-1}(x)}_y$, we can write
\begin{equation}
\Tr^{E}[F(x)\,\Pi_{\pi}(y)]\rho_p(y)
=\Tr^{E}\left[F(x)\,\left(\ev^{\pi^{-1}(x)}_y\right)^*
\ev^{\pi^{-1}(x)}_y\rho_{E_p}(x)\right]\,.
\end{equation}
Via formula \eqref{nupidef},
this implies in particular that for any $f\in\cinf(M,\R)$,
we have
\begin{multline}
\<f,\Tr^{E}[\left(F\circ\pi\right)\Pi_\pi]\>_{W_p}\\
=\int_{x\in X}\,\int_{y\in\pi^{-1}(x)}
f(y) \Tr^{E_p}\left[\left(\ev^{\pi^{-1}(x)}_y\right)^*
\ev^{\pi^{-1}(x)}_y\rho_{E_p}(x)F(x)\right]
\,d\nu_{\pi^{-1}(x)}(y)d\nu_X(x)\\
=\<T_{\pi}(f),F\>_{W_{E_p}}\,.
\end{multline}
This concludes the proof.
\end{proof}

\subsection{Berezin-Toeplitz quantization of vector bundles}
\label{BTvbsec}

Let $(E,h^E)$ be a holomorphic Hermitian vector bundle
over $X$, and fix
a smooth volume form $d\nu_X$ on $X$. Recall that for any
$p\in\N$, we set $E_p:=E\otimes L^p$ with induced Hermitian metric
$h^{E_p}$, and we
write $\HH_p$ for the space of holomorphic sections of $E_p$,
endowed with the $L_2$-Hermitian product $\<\cdot,\cdot\>_p$ defined
by formula \eqref{L2E}.
In this section, we discuss the semi-classical properties
of the Berezin-Toeplitz
quantization of $(E,h^{E})$ over $(X,p\,\om)$ as $p\to+\infty$.

The following basic result first shows that Definition
\ref{BTquantdefE} coincides
with the
usual Berezin-Toeplitz quantization of vector bundles of \cite{MM08b}.
\begin{prop}\label{BTquantfla}
For any $F\in\cinf(X,\Herm(E_p))$,
the restriction to $\HH_p\subset\cinf(X,E_p)$
of the operator acting on $L_2$-sections of $E_p$
by the formula
\begin{equation}\label{PfP}
T_{E_p}(F):=P_{E_p}\,F\,P_{E_p}\,,
\end{equation}
where $P_{E_p}:\cinf(X,E_p)\fl\HH_p$ is the orthogonal
projection on holomorphic sections with respect to $\<\cdot,\cdot\>_p$
and where $F$ acts
pointwise on smooth sections $\cinf(X,E_p)$,
coincides with the Berezin-Toeplitz quantization of Definition \ref{BTquantdefE}.
\end{prop}
\begin{proof}
For any $s_1,\,s_2\in \HH_p$, the Berezin-Toeplitz quantization
of $F\in\cinf(X,\Herm(E_p))$ satisfies
\begin{equation}
\begin{split}
\<T_{E_p}(F)\,s_1,s_2\>_p&=
\int_X\<(\ev_x^E)^*\,F(x)\,\ev_x^E\,s_1,s_2\>_p\,d\nu_X(x)\\
&=\int_X h^{E_p}(F(x)\,s_1(x),s_2(x))\,d\nu_X(x)\\
&=\<P_{E_p}\,F\,P_{E_p}\,s_1,s_2\>_p\,.
\end{split}
\end{equation}
This gives formula \eqref{PfP}.
\end{proof}

Let
$\nabla^{\End(E)}$ denote the Chern connection of
$\End(E)$ endowed with the Hermitian metric induced by $h^E$,
and for any $m\in\N$,
write $|\cdot|_{\CC^m}$ for the associated local
$\CC^m$-norm on $\cinf(X,\End(E))$.
Write $\dbar$ and $\partial$ for the $(0,1)$ and
$(1,0)$-parts of $\nabla^{\End(E)}$, and recall that we write
$g^{TX}$ for the Riemannian metric \eqref{gTX} on $X$.

\begin{theorem}\cite[Th.\,0.3]{MM12}\label{BMSvb}
For any $F\in\cinf(X,\Herm(E))$, we have
\begin{equation}
\|T_{E_p}(F)\|_{op}\xrightarrow{p\to+\infty}|F|_{\CC^0}\,.
\end{equation}
Furthermore, for any $F,\,G\in\cinf(X,\Herm(E))$, we have the following
estimate in operator norm as $p\to+\infty$,
\begin{equation}\label{BMSvbfla}
[T_{E_p}(F),T_{E_p}(G)]=
T_{E_p}([F,G])+\frac{\sqrt{-1}}{2\pi p}T_{E_p}(C(F,G))+O(p^{-2})\,,
\end{equation}
with
\begin{equation}\label{C1}
C(F,G):=\sqrt{-1}\left(\<\partial F,
\dbar G\>_{g^{TX}}
-\<\partial G,\dbar F\>_{g^{TX}}\right)\,.
\end{equation}
\end{theorem}

This rest of the Section is dedicated to the statements of refined
semi-classical properties
of the Berezin-Toeplitz
quantization of $(E,h^{E})$ over $(X,p\,\om)$ as $p\to+\infty$,
taken from \cite{DLM06} and \cite{MM12}, which lie at the core of
the applications of quantization in stages in the next sections.
Writing $\pi_j:X\times X\fl X$, with $j=1,\,2$, for
the first and second projections, and setting
\begin{equation}
E_p\boxtimes E_p^*:=\pi_1^*E_p\otimes\pi_2^*E_p^*
\end{equation}
as a holomorphic Hermitian vector bundle over $X\times X$,
these refined semi-classical properties involve the
following fundamental notion.

\begin{defi}\label{Bergdef}
The \emph{Bergman kernel} of $(E_p,h^{E_p})$ over $X$ is the
Schwartz kernel of the orthogonal
projection $P_{E_p}:\cinf(X,E_p)\fl\HH_p$
with respect to the $L_2$-Hermitian product \eqref{L2E},
characterized as a section
$P_{E_p}(\cdot,\cdot)\in\cinf(X\times X,E_p\boxtimes E_p^*)$
for any $s\in\cinf(X,E_p)$ and $x\in X$
by the formula
\begin{equation}\label{ker}
P_{E_p}\,s\,(x):=\int_X P_{E_p}(x,y)s(y)\,d\nu_X(y)\,.
\end{equation}
\end{defi}
For all $x\in X$ and $p\in\N$, recall that
we write $\ev_{E_x}^*:E_{p,x}\fl \HH_p$ for the
dual of the evaluation map \eqref{evE} with respect to
$h^{E_p}$ and $\<\cdot,\cdot\>_p$.

\begin{lem}\label{Rplem}
For any $p\in\N,\,y\in X$ and $v\in E_{p,y}$, the holomorphic
section $\ev_{E_y}^*.v\in\HH_p$ satisfies the following
formula, for all $x\in X$,
\begin{equation}\label{ev=Pp}
\ev_{E_y}^*.v\,(x)=P_{E_p}(x,y).v\,.
\end{equation}
Furthermore, for any $p\in\N$,
the Rawnsley section $\rho_{E_p}\in\cinf(X,\Herm(E_p))$ of Definition
\ref{Rawnsect}
satisfies the following formula, for any $x\in X$,
\begin{equation}\label{Rp=Pp}
\rho_{E_p}(x)=P_{E_p}(x,x)\,.
\end{equation}
\end{lem}
\begin{proof}
By definition, the dual of the evaluation \eqref{evE} at $y\in X$
is characterized by the following formula,
for all $v\in E_{p,y}$ and $s\in\HH_p$, 
\begin{equation}\label{characcoh}
\left\langle s,\ev_{E_y}^*.v\right\rangle_p
=h^{E_p}(s(y),v)\,.
\end{equation}
On the other hand, using the characterization \eqref{ker} of the
Bergman kernel as well as the formula $P(x,y)=P(y,x)^*$,
holding for the Schwartz kernel of any self-adjoint operator, we have
\begin{equation}
\begin{split}
\int_X\,h^{E_p}(s(x),P_{E_p}(x,y).v)\,d\nu_X(x)&=
\int_X\,h^{E_p}(P_{E_p}(y,x)s(x),v)\,d\nu_X(x)\\
&=h^{E_p}(P_{E_p}\,s(y),v)=h^{E_p}(s(y),v)\,.
\end{split}
\end{equation}
This proves formula \eqref{ev=Pp}. This readily implies formula
\eqref{Rp=Pp} by Definition \ref{Rawnsect} of the Rawnsley section,
as for all $x\in X$ and $v\in E_{p,x}$ we have
\begin{equation}
\rho_{E_p}(x).v=\ev_{E_x}\ev_{E_x}^*.v=\ev_{E_x}^*.v\,(x)
=P_{E_p}(x,x).v\,.
\end{equation}
\end{proof}

Let us assume from now on that $\nu_X$ is the Liouville measure.
Note that there is no loss in generality with this assumption, as
one can always accomodate this change by multiplying the Hermitian
metric $h^E$ by a scalar function. Such an operation leaves
unchanged the induced Hermitan metric on $\End(E)$, so that
all the results of this Section and the next are valid without any
modification in the general case.

For any holomorphic Hermitian vector bundle $(E,h^E)$,
recall that we write $R^E\in\Om^2(X,\End(E))$ for the curvature of its Chern
connection.
Let $K_X=\det(T^{(1,0)}X)^*$ be the \emph{canonical line bundle} of $X$
endowed with the Hermitian metric $h^{K_X}$ induced by $g^{TX}$.
The \emph{Ricci form} $\textup{Ric}(\om)\in\Om^2(X,\R)$ of
$(X,J,\om)$ is defined by the formula
\begin{equation}\label{Ricdef}
\Ric(\om):=-\sqrt{-1}R^{K_X}\,,
\end{equation}
and the \emph{scalar curvature} $\scal(\om)\in\cinf(X,\R)$ of $g^{TX}$
can be defined by the formula
\begin{equation}\label{scalcurvdef}
\scal(\om):=\<\om,\Ric(\om)\>_{g^{TX}}\,.
\end{equation}

\begin{theorem}\label{Bergdiagexp}
{\cite[Th.\,1.3]{DLM06}}
There exist Hermitian endomorphisms
$b_0,\,b_1,\,b_2\in\cinf(X,\End(E))$
such that for any $m\in\N$, there exists $C_m>0$ and $l\in\N$
such that for all $p\in\N$ big enough,
\begin{equation}
\left|\frac{1}{p^n}\rho_{E_p}-\left(b_0+\frac{1}{p}\,b_1+\frac{1}{p^2}b_2
\right)\right|_{\CC^m}\leq \frac{C_m}{p^3}\,,
\end{equation}
uniformly in the $\CC^m$-norm of the derivatives of
$h$ and $h^E$ up to order $l$. Furthermore, we have
\begin{equation}\label{Bergcoeff}
b_0=\Id_E\quad\text{and}\quad b_1=\frac{\scal(\om)}{8\pi}\Id_E+
\frac{\sqrt{-1}}{2\pi}\<\om,R^E\>_{g^{TX}}\,.
\end{equation}
\end{theorem}
Lemma \ref{Rplem} and Theorem \ref{Bergdiagexp} imply in particular
that as $p\fl+\infty$, we have
\begin{equation}\label{npfla}
\dim\HH_p=\int_X\,\Tr^{E_p}[\rho_{E_p}(x)]\,d\nu_X(x)=p^n\Vol_h(X)\rk(E)+O(p^{n-1})\,,
\end{equation}
which can also be seen as a consequence of the classical Hirzebruch-Riemann-Roch
formula. For any $F\in\cinf(X,\Herm(E))$, consider the operator
$T_{E_p}(F):=P_{E_p}\,F\,P_{E_p}$ acting on $\cinf(X,E_p)$ as in
Proposition \ref{BTquantfla}, so that it coincides with the Berezin-Toeplitz
quantization of $F$ of Definition \ref{BTquantdefE}
when restricted to $\HH_p$, and write
$T_{E_p}(F)(\cdot,\cdot)\in\cinf(X\times X,E_p\boxtimes E_p^*)$
for its Schwartz kernel. By the basic composition formula
for operators with smooth Schwartz kernels,
for all $x,\,y\in X$, we have
\begin{equation}
T_{E_p}(F)(x,y)=\int_X\,P_{E_p}(x,w)F(w)P_{E_p}(w,y)\,d\nu_X(w)\,.
\end{equation}
Let $\<\cdot,\cdot\>_{L^2}$ be the $L_2$-Hermitian product on
$\End(E)$ induced by $h^E$ and the Liouville measure $\nu_X$,
write $\|\cdot\|_{L_2}$ for the associated norm
and write $L_2(X,\End(E))$ for the induced space of square-integrable
sections of $\End(E)$.
The \emph{Bochner Laplacian} $\Delta$
is the second order
differential operator characterized
for any $F_1,\,F_2\in\cinf(X,\End(E))$ with support in a local
chart by the formula
\begin{equation}\label{delta}
\<\Delta F_1,F_2\>_{L_2}=\sum_{j=1}^{2n}
\int_X\left\langle\nabla^{\End(E)}F_1(x),
\nabla^{\End(E)}F_2(x)\right\rangle_{\End(E)\otimes T^*X}\,d\nu_X(x)\,,
\end{equation}
where the pairing on $\End(E)\otimes T^*X$ with values in $\C$
is the one induced by
$h^E$ and $g^{TX}$.

\begin{theorem}\label{Toepdiagexp}
{\cite[Th.\,0.1,\,(0.13)]{MM12}}
For any $F\in\cinf(X,\End(E))$, there exist Hermitian endomorphisms
$b_{0}(F),\,b_{1}(F)\in\cinf(X,\End(E))$
such that for any $m\in\N$, there exists $C_m>0$ and $l\in\N$
such that for any $x\in X$ and all $p\in\N$ big enough,
\begin{equation}
\left|\frac{1}{p^n}T_{E_p}(F)(x,x)
-\left(F(x)+\frac{1}{p}\,b_{1}(F)(x)+\frac{1}{p^2}\,b_{2}(F)(x)\right)
\right|_{\CC^m}\leq \frac{C_m}{p^3}\,,
\end{equation}
uniformly in the $\CC^m$-norm of the derivatives of
$F,\,h$ and $h^E$ up to order $l$. Furthermore, we have
\begin{equation}\label{Toepcoeff}
b_{1}(F)=\frac{\scal(\om)}{8\pi}F+
\frac{\sqrt{-1}}{4\pi}(\<\om,R^E\>_{g^{TX}}\,F+F\,\<\om,R^E\>_{g^{TX}})
-\frac{1}{4\pi}\Delta F\,,
\end{equation}
where $\<\cdot,\cdot\>_{g^{TX}}$
denotes the pairing on $\Om^2(X,\End(E))$ with values in $\End(E)$
induced by $g^{TX}$.

Finally, in case $(E,h^E)$ is the trivial line bundle,
for all $f\in\cinf(X,\R)$ we have
\begin{equation}
b_{2}(f)=b_2\,f+\frac{\Delta^2}{32\pi^2}f
-\frac{\scal(\om)}{32\pi^2}\Delta f-\frac{\sqrt{-1}}{8\pi^2}
\<\Ric(\om),\partial\dbar f\>_{g^{TX}}\,,
\end{equation}
where $b_2\in\cinf(X,\C)$ is as in Theorem \ref{Bergdiagexp}
for $E=\C$.
\end{theorem}

In the context of Section
\ref{stagesec},
it is more natural to consider instead the Laplacian \eqref{KodLapdef},
whose relation with the Bochner Laplacian \eqref{delta} is given by the
following \emph{Weitzenböck formula}, which can be found in
\cite[Prop.\,1.2]{Bis87}.

\begin{prop}\label{Weitzenfla}
For any $F\in\cinf(X,\End(E))$, the following identity
holds,
\begin{equation}
\Box\,F=\Delta F
-\sqrt{-1}\left[\<\om,R^E\>_{g^{TX}},F\,\right]\,.
\end{equation}
\end{prop}

\section{Spectral estimates for Berezin transforms}
\label{specsec}


In this Section, we use the setting of quantization in stages developed
in Section \ref{BTsec} to extend the study of Berezin-Toeplitz quantization
from the point of view of quantum measurement in \cite{IKPS19}
to the case of vector bundles. In particular,
we introduce a natural notion of a \emph{Berezin transform} in this context,
and establish asymptotic estimates as $p\to+\infty$ on
its \emph{spectral gap} in the manner of \cite{IKPS19}. We then apply
these estimates to Donaldson's iterations towards $\nu$-balanced metrics
on stable vector bundles.

\subsection{Berezin transform on vector bundles}
\label{specbersec}

Recall the notations of Section \ref{BTvbsec}.
In this Section, we introduce the
\emph{Berezin transform} in the context of quantization in stages,
which is a key tool for the study of the
quantum-classical correspondence for Berezin-Toeplitz quantization.
The main goal of this Section is to give a proof of Theorem \ref{gapvb}.

Recall from Proposition \ref{BersymbE} that we write
$T^*_{E_p}:\Herm(\HH_p)\to\cinf(X,\Herm(E))$ for the dual of
the Berezin-Toeplitz quantization map
$T_{E_p}:\cinf(X,\Herm(E))\to\Herm(\HH_p)$ with respect to the scalar
product \eqref{L2alphaE}.

\begin{defi}\label{BpdefE}
The \emph{Berezin transform} of $(E,h^E)$ over $(X,p\,\om)$
is the linear operator acting on $\cinf(X,\Herm(E))$ defined by
\begin{equation}
\cB_{E_p}:=T_{E_p}^*\circ T_{E_p}:\cinf(X,\Herm(E))
\longrightarrow\cinf(X,\Herm(E))\,.
\end{equation}
\end{defi}
Definition \ref{BpdefE} naturally extends the definition
of a Berezin transform given in \cite[\S\,2]{IKPS19} in the case
of $E=\C$. In particular, it retains most of the
essential properties of a \emph{Markov operator}.
Namely, it is by definition a positive self-adjoint operator
with respect to the- $L_2$-scalar product \eqref{L2alphaE}, and since
it factorizes through the finite dimensional vector space $\HH_p$,
it has finite image, so that it admits a discrete spectrum inside
$[0,+\infty)$ and its positive eigenvalues have finite multiplicity.

The following Proposition describes the behavior of the Berezin transform
under quantization in stages, where $(E,h^E)$ over $X$ is the quantum-classical hybrid of a prequantized Kähler fibration
$\pi:(M,\om^M)\to (X,\om)$ as in Definition \ref{preqfib}.

\begin{prop}\label{BpdefEstage}
The Berezin transform of Definition \ref{BpdefE} 
is characterized by the following formula, for all $f\in\cinf(M,\R)$,
\begin{equation}
\<\cB_{E_p}T_\pi(f),T_\pi(f)\>_{W_{E_p}}=\<\cB_p\,f,f\>_{W_p}\,,
\end{equation}
where $\cB_p:\cinf(M,\R)\fl\cinf(M,\R)$ is the Berezin transform
of $\cL\otimes\pi^*L^p$ over $M$, and $T_{\pi}(f)\in\cinf(X,\Herm(E_p))$
is the Berezin-Toeplitz
quantization of $f\in\cinf(M,\R)$ in the fibres of $\pi:M\fl X$.
\end{prop}
\begin{proof}
This is a straightforward consequence of Proposition \ref{Tfunct} and
of the definition of $T^*_{E_p}$ as the dual of $T_{E_p}$
with respect to the $L_2$-Hermitian product \eqref{L2alphaE}.
\end{proof}

The following basic result explains the role played by the asymptotic
expansions of Theorems \ref{Bergdiagexp} and \ref{Toepdiagexp}
in the proof Theorem \ref{mainthvb}.

\begin{prop}\label{Bpflaprop}
The Berezin transform of Definition \ref{BpdefE} satisfies
the following formula, for all $F\in\cinf(X,\Herm(E))$ and
$x\in X$,
\begin{equation}\label{Bpfla}
\cB_{E_p}(F)(x)\rho_{E_p}(x)=T_{E_p}(F)(x,x)\,.
\end{equation}
\end{prop}
\begin{proof}
Recall the definition \eqref{L2alphaE} of
the $L_2$-scalar product $\<\cdot,\cdot\>_{W_{E_p}}$.
By Proposition \ref{BTquantfla} and Definition \ref{BpdefE},
using $P_{E_p}P_{E_p}=P_{E_p}$ and the basic trace formula for
operators with smooth Schwartz kernels, for any
$F_1,\,F_2\in\cinf(X,\End(E))$ we have
\begin{equation}
\begin{split}
\<\cB_{E_p}&(F_1)\rho_{E_p},F_2\>_{W_{E_p}}
=\<\cB_{E_p}(F_1),\rho_{E_p}F_2\>_{W_{E_p}}\\
&=\<\<T_{E_p}(F_1),T_{E_p}(\rho_{E_p}F_2)\>\>
=\Tr^{\HH_p}\left[P_{E_p}F_1P_{E_p}\rho_{E_p}F_2\right]\\
&=\int_X\int_X \Tr^{E_p}\left[P_{E_p}(x,y)F_1(y)
P_{E_p}(y,x)\rho_{E_p}(x)F_2(x)\right]
\,d\nu_X(x)d\nu_X(y)\\
&=\int_X \Tr^{E_p}\left[T_{E_p}(F_1)(x,x)
\rho_{E_p}(x)F_2(x)\right]
\,d\nu_X(x)\,.
\end{split}
\end{equation}
This proves formula \eqref{Bpfla}.
\end{proof}
%


Let us now describe the proof of Theorem \ref{mainthvb}, following the
proof of the analogous result for the scalar Berezin transform
in \cite[\S\,3]{IKPS19}.
The following Proposition is an extension of the refined
Karabegov-Schlichenmaier formula of \cite[Prop.\,3.8]{IKPS19}.

\begin{prop}\label{KSexpprop}
For any $m\in\N$, there exists $l\in\N$ and a constant
$C_m>0$, uniform in the
$\CC^m$-norm of the derivatives of $h^L$ and $h^E$ up to order $l$,
such that for any
$F\in\cinf(X,\End(E))$ and all $p\in\N$ big enough, we have
\begin{equation}\label{KSexp}
\left|\,\cB_{E_p}F-F+\frac{\Box}
{4\pi p} F\,\right|_{\CC^m}\leq
\frac{C_m}{p^2}|F|_{\CC^{m+4}}\;.
\end{equation}
\end{prop}
\begin{proof}
Applying Theorems \ref{Bergdiagexp} and \ref{Toepdiagexp}
to Proposition \ref{Bpflaprop},
and following the proof of the analogous result for $E=\C$
in \cite[Prop.\,3.8]{IKPS19}, we deduce that for all $m\in\N$,
there is $C_m>0$ such that for all $f\in\cinf(X,\End(E))$, we have
\begin{equation}\label{KSexpab}
\left|\cB_{E_p}F-F-\frac{1}{p}D_2\,F
\right|_{\CC^m}\leq
p^{-2}C_m|F|_{\CC^{m+4}}\;,
\end{equation}
where $D_2$ is a second order differential operator satisfying
\begin{equation}\label{D2comput}
\begin{split}
D_2\,F&=b_{1}(F)-Fb_1\\
&=\frac{\sqrt{-1}}{4\pi}(\<\om,R^E\>_{g^{TX}}\,F+F\,\<\om,R^E\>_{g^{TX}})-\frac{1}{4\pi}\Delta F-\frac{\sqrt{-1}}{2\pi}
F\<\om,R^E\>_{g^{TX}}\\
&=-\frac{1}{4\pi}\Delta F+
\frac{\sqrt{-1}}{4\pi}(\<\om,R^E\>_{g^{TX}}\,F-F\,\<\om,R^E\>_{g^{TX}})
\,.
\end{split}
\end{equation}
This gives formula \eqref{KSexp} via the Weitzenböck formula
of Proposition \ref{Weitzenfla}.
\end{proof}

Recall the increasing sequence \eqref{lapev}
of eigenvalues of the Laplacian $\Box$, and
for all $j\in\N$, let $e_j\in\cinf(X,\End(E))$ be the normalized
eigensection associated with $\lambda_j^E$,
so that
\begin{equation}
\|e_j\|_{L_2}=1\quad\text{and}\quad\Box\,e_j=\lambda_j^E e_j\,.
\end{equation}
For a function $\Psi:\R\fl\R$ with at most polynomial growth at
$+\infty$, we define the operator $\Psi(\Box)$
acting on $F\in\cinf(X,\End(E))$ by the formula
\begin{equation}\label{calculfct}
\Psi(\Box)F=\sum_{i=0}^{+\infty}
\Psi(\lambda_j)\<F,e_j\>_{L_2}e_j\;.
\end{equation}
Using the functional calculus \eqref{calculfct},
we write $\|\cdot\|_{H^m}$ for the
Sobolev norm of order $m\in\N$, defined for any $f\in\cinf(X,\C)$by
\begin{equation}\label{ellest}
\|F\|_{H^m} := \|\Delta_h^{m/2}F\|_{L_2} + \|F\|_{L_2}\;,
\end{equation}
For any $t>0$, we write
$\exp(-t\Box)$ for the \emph{heat operator} associated
with the Laplacian \eqref{KodLapdef} acting on $\cinf(X,\End(E))$.
For any $m\in\N$, let $\|\cdot\|_{H^m}$ be a
Sobolev norm of order $m$ on $\cinf(X,\End(E))$.

\begin{prop}\label{boundexp}
For any $m\in\N$, there exist $l\in\N$ and a constant $C_m>0$,
uniform in the
$\CC^m$-norm of the derivatives of $h$ and $h^E$ up to order $l$,
such that for any
$F\in\cinf(X,\End(E))$ and all $p\in\N$, we have
\begin{equation}\label{boundexpfla}
\left\|\,\exp\Big(-\frac{\Box}{4\pi p}\Big)\,
F-\cB_p(F)\,\right\|_{H^m}\leq
\frac{C_m}{p}\|F\|_{H^m}\;.
\end{equation}
\end{prop}
\begin{proof}
In the same way than in the proof of the analogous
result in \cite[Prop.\,3.9]{IKPS19},
this readily follows from the off-diagonal expansion of the Bergman
kernel of Definition \ref{Bergdef} as $p\fl+\infty$
given in \cite[Th.\,4.18']{DLM06} and
the classical asymptotic expansion of heat kernels for
generalized Laplacians as $t\fl 0$,
which can be found for example in \cite[Th.\,2.29]{BGV04}.
\end{proof}


The following key result, inspired by a strategy of Lebeau and Michel in
\cite{LM10}, allows to control the eigenvalues of $\cB_p$
without controlling its $L_2$-norm.

\begin{prop}\label{hjrefpE}
For any $L>0$ and $m\in\N$, there exist constants $C_{m}>0$ and
$p_m\in\N$,
uniform in the $\CC^m$-norm of the derivatives of
$h$ up to some finite order,
such that for any $p\geq p_m,\,\mu\in\C$ and $F\in\cinf(X,\End(E))$
satisfying
\begin{equation}\label{Sp=mufE}
\cB_{E_p}F=\mu F~~\text{and}~~p|1-\mu|<L\,,
\end{equation}
we have
\begin{equation}\label{hjrefinedE}
\|F\|_{H^m}\leq C_{m}\|F\|_{L_2}\;.
\end{equation}
\end{prop}
\begin{proof}
Using Propositions \ref{KSexpprop} and \ref{boundexp}, the result follows
from a straightforward extension of the proof of
\cite[Th.\,3.1,\,\S\,3.5]{IKPS19}.
\end{proof}

\subsection{Kähler-Einstein case and proof of Theorem \ref{mainthvb}}
\label{BTKEsec}

We now consider the
scalar case $E=\C$, and when
the Kähler metric $g^{TX}$ defined by equation \eqref{gTX}
is \emph{K\"ahler-Einstein}, so that there exists a constant
$c\in\R$ such that
\begin{equation}\label{KEdef}
\textup{Ric}(\om)=c\om\,,
\end{equation}
where the \emph{Ricci form} $\textup{Ric}(\om)\in\Om^2(X,\R)$
is defined via formula \eqref{Ricdef}.
We also assume that $\nu_X$ is the
Liouville measure.
Recall that in the scalar case $E=\C$, the Laplacian \eqref{KodLapdef}
coincides with the usual Laplace-Beltrami operator $\Delta$ of $(X,g^{TX})$
acting on $\cinf(X,\C)$.
The first step of the proof of the refined estimate
\ref{mainthKE} is the following
refinement of
the Karabegov-Schlichenmaier formula of Proposition \ref{KSexpprop}.

\begin{prop}\label{KS_KE}
For any $m\in\N$, there exists $C_m>0$ such that for any
$f\in\cinf(X,\C)$ and all $p\in\N$, we have
\begin{equation}\label{KSexp_KE}
\left|\,\cB_pf-\left(1-\frac{\Delta}{4\pi p}
+\frac{\Delta^2}{32\pi^2p^2}+c\frac{\Delta}{16\pi^2p^2}\right)f
\,\right|_{\CC^m}\leq
\frac{C_m}{p^3}|f|_{\CC^{m+4}}\;.
\end{equation}
\end{prop}
\begin{proof}
Recall from Proposition \ref{Bpflaprop} that for any $x\in X$ and $p\in\N$
big enough, we have
\begin{equation}\label{Bpdiagscal}
\cB_{p}f(x)=\frac{T_{p}(f)(x,x)}{\rho_{p}(x)}\,.
\end{equation}
Thus by Theorems \ref{Bergdiagexp} and \ref{Toepdiagexp}
as in the proof of Proposition \ref{KSexpprop},
we know that for all $m\in\N$,
there is $C_m>0$ such that
\begin{equation}\label{KSexpab_KE}
\left|\cB_pf-\left(1+\frac{D_2}{p}+\frac{D_4}{p^2}\right) f
\right|_{\CC^m}\leq
\frac{C_m}{p^3}|f|_{\CC^{m+6}}\;,
\end{equation}
where the differential operator $D_2$
has been computed in equation \eqref{D2comput} for $E=\C$
to satisfy
\begin{equation}\label{D2}
D_2\,f=-\frac{\Delta}{4\pi p}f\,.
\end{equation}
while the differential operator $D_4$ satisfies
\begin{equation}\label{D4pre}
\begin{split}
D_4\,f&=b_{2}(f)-b_2 f-b_1(b_{1}(f)-b_1 f)\\
&=\frac{\Delta^2}{32\pi^2}f
-\frac{\scal(\om)}{32\pi^2}\Delta f-\frac{\sqrt{-1}}{8\pi^2}
\<\Ric(\om),\partial\dbar f\>_{g^{TX}}+\frac{\scal(\om)}{8\pi}
\frac{\Delta f}{4\pi}\\
&=\frac{\Delta^2}{32\pi^2} f-\frac{\sqrt{-1}}{8\pi^2}
\<\Ric(\om),\partial\dbar f\>_{g^{TX}}\,.
\end{split}
\end{equation}
Now by definition \eqref{KEdef}
of a Kähler-Einstein metric
and of the holomorphic Laplacian of $(X,g^{TX},J)$, which is half
the Laplace-Beltrami operator $\Delta$, we know that
\begin{equation}\label{omddbar=delta}
\sqrt{-1}\<\Ric(\om),\partial\dbar f\>_{g^{TX}}=
c\sqrt{-1}\<\om,\partial\dbar f\>_{g^{TX}}=-\frac{c}{2}\Delta f\,.
\end{equation}
This gives the result.
\end{proof}

The crucial point in the asymptotic expansion \eqref{KSexp_KE}
of the Berezin transform is the fact
that it is a polynomial in the Laplace-Beltrami
operator $\Delta$.
This is a consequence of the Kähler-Einstein hypothesis, and
is at the core of the following proof of Theorem \ref{mainthvb},
which is a refinement of \cite[\S\,3.5]{IKPS19}.

\medskip\noindent{\bf Proof of Theorem \ref{mainthvb}.}
Let us start with the more involved estimate \eqref{mainthKE}.
Recall the $L^2$-norm $\|\cdot\|_{L_2}$
on $\cinf(X,\C)$
induced by the Riemannian measure of $g^{TX}$.
By Proposition \ref{KS_KE} and by the classical
Sobolev embedding theorem,
for any $f\in\cinf(X,\C)$ we get
\begin{equation}\label{KSSob}
\left\|p(1-\cB_p)f-a_p(\Delta)
f\right\|_{L_2}
\leq Cp^{-2}\|f\|_{H^{m}}\;,
\end{equation}
for some $m\in\N$ large enough, where $a_p\in\C[x]$ is a
polynomial defined for all
$p\in\N$ by
\begin{equation}\label{apl}
a_p(x)=\frac{x}{4\pi}
-\frac{x^2+2cx}{32\pi^2p}\,.
\end{equation}
Recall now that by Definition
\ref{BpdefE} with $E=\C$, the Berezin transform $\cB_p$ is
a self-adjoint with respect to the $L_2$-scalar product \eqref{L2alphaE}.
Writing $\|\cdot\|_{W_p}$ for the associated $L_2$-norm,
Theorem \ref{Bergdiagexp} shows that there exists a constant $C>0$
such that
\begin{equation}\label{WpL2eq}
\left(1-\frac{C}{p}\right)\|\cdot\|_{L^2}\leq
\|\cdot\|_{W_p}
\leq \left(1+\frac{C}{p}\right)\|\cdot\|_{L^2}\,.
\end{equation}
This shows in particular that equation \eqref{KSSob} also holds
in the norm $\|\cdot\|_{W_p}$.
Letting now $j\in\N$ be fixed and $e_j\in\cinf(X,\C)$ satisfy
$\Delta e_j=\lambda_j e_j$ and $\|e_j\|_{L_2}=1$,
estimate \eqref{KSSob} reads
\begin{equation}\label{estevdelt}
\left\|p(1-\cB_p)e_j-a_p(\lambda_j)e_j\right\|
_{W_p}\leq C p^{-2}\;.
\end{equation}
Thus if $m_j\in\N$ is the multiplicity of $\lambda_j$ as an
eigenvalue of $\Delta$, the estimate \eqref{estevdelt} for all
eigenfunctions of $\Delta$ associated with $\lambda_j$
gives a constant $C_j>0$ such that for all $p\in\N$,
\begin{equation}\label{spec>}
\#\left(\Spec\big(p(1-\cB_p)\big)\cap\left[a_p(\lambda_j)-C_jp^{-2},a_p(\lambda_j)+C_jp^{-2}\right]\right)
\geq m_j\;.
\end{equation}
This immediately follows from the variational principle applied to
the compact operator $p(1-\cB_p)$.

Conversely, fix $L>0$ and let $\{f_p\}_{p\in\N}$ be
the sequence of normalized eigenfunctions considered
in Proposition \ref{hjrefpE} for $E=\C$. Then by
\eqref{KSSob}, we get $C>0$ such that
\begin{equation}\label{specinv>}
\left\|p(1-\mu_p)f_p-a_p(\Delta)
f_p\right\|_{L_2}\leq Cp^{-2}\;.
\end{equation}
In particular, we get that
\begin{equation}\label{eq-spec-est}
\textup{dist}\left(p(1-\mu_p),\Spec\,a_p(\Delta)
\right)\leq Cp^{-2}\;,
\end{equation}
showing that all eigenvalues of $p(1-\cB_p)$ bounded by some
$L>0$ have to be included in the left hand side of \eqref{spec>}.

Let us finally show that \eqref{spec>} is an equality
for $p\in\N$ big enough. Let $l\in\N$ with $l\geq m_j$
be such that for all $p\in\N$, there exists
an orthonormal family $\{f_{k,p}\}_{1\leq k\leq l}$
of eigenfunctions of $\cB_p$ for $\|\cdot\|_{W_p}$ with associated
eigenvalues $\{\mu_{k,p}\in\R\}_{1\leq k\leq l}$
satisfying
\begin{equation}
p(1-\mu_{k,p})\in[a_p(\lambda_j)-Cp^{-2},a_p(\lambda_j)+Cp^{-2}]\;,~~
\text{for all}~~1\leq k\leq l\;.
\end{equation}
As the inclusion of the Sobolev space $H^q$ in $H^{q-1}$ is compact,
by Proposition \ref{hjrefpE} and \eqref{WpL2eq},
there exists a subsequence of $\{f_{k,p}\}_{p\in\N}$ converging
to a function $f_k$ in $H^{q-1}$-norm,
for all $1\leq k\leq l$.
In particular, taking $q>2$ and using \eqref{WpL2eq} again,
the family $\{f_k\}_{1\leq k\leq l}$
is orthonormal in $L_2(X,\C)$
and satisfies $\Delta f_k=\lambda_j f_k$ for all $1\leq k\leq l$ by
\eqref{specinv>}. By definition of the multiplicity $m_j\in\N$
of $\lambda_j$ and as $a_p:\R\fl\R$ is strictly increasing over $[0,L]$
for all $p\in\N$ big enough, this forces $l=m_j$. We thus get
\begin{equation}\label{spec=}
\#\left(\Spec\big(p(1-\cB_p)\big)\cap
\left[a_p(\lambda_j)-Cp^{-2},a_p(\lambda_j)+Cp^{-2}\right]\right)
=m_j\;.
\end{equation}
Using again the fact that $a_p:\R\fl\R$ is increasing over $[0,L]$
for all $p\in\N$ big enough, so that the respective order of the eigenvalues
\eqref{berev} and \eqref{lapev} is respected, this establishes
the estimate \eqref{mainthKE}.
%
On the other hand, the estimate
\eqref{eq-LB} follows from Propositions
\ref{KSexpprop} and \ref{hjrefpE} in the same way.

\qed

\section{Balanced metrics}
\label{balsec}

The goal
of this Section is to establish Theorem \ref{gapvb} on the
exponential convergence of the iterations of the dynamical systems
\eqref{Tdefcan} and \eqref{DonitflaEintro} towards the respective
notions of balanced metrics, due to Donaldson in \cite{Don01}
in the scalar case and Wang in \cite{Wan05} in the vector bundle case.
We then establish in Section \ref{momsec} remarkable identities
relating the dynamical systems \eqref{DonitflaEintro} and \eqref{Tdefcan}
with the \emph{moment map} for balanced
embeddings used in \cite{Wan05} and \cite{Don01}, recovering the
lower bounds which played a crucial
role in the proofs of the celebrated results of Wang \cite{Wan05}
and Donaldson \cite{Don01} on balanced embeddings.

\subsection{Vector bundle case}
\label{vbcase}

Consider now $E$ as a holomorphic vector bundle over $(X,J)$, and
fix a smooth volume form $d\nu$ over $X$.
For any $p\in\N$, we write $\HH_p$ for the space
of holomorphic sections of $E_p=E\otimes L^p$.
Let $\Met(E_p)$ be the space
of Hermitian metrics on $E_p$ and $\Prod(\HH_p)$ the
space of inner Hermitian products on $\HH_p$.
Then the \emph{Hilbert map} of $E_p$ associated with $\nu$ is defined by
\begin{equation}\label{HilbEfla}
\begin{split}
\Hilb_{E_p}:\Met(E_p)&\longrightarrow\Prod(\HH_p)\\
h^{E_p}\,&\longmapsto
\frac{\dim\HH_p}{\Vol(X,\nu)\rk(E)}\,\int_X\,h^{E_p}(\cdot,\cdot)\,d\nu
\,,
\end{split}
\end{equation}
Recall by Kodaira's vanishing theorem that there exists $p_0\in\N$
such that for all $p\geq p_0$, the evaluation map \eqref{evE}
associated with $E_p=E\otimes L^p$
is surjective for all $x\in X$. For
any $p\geq p_0$, the \emph{Fubini-Study map} associated with
$E_p$ is the map
\begin{equation}\label{FSdef}
\FS:\Prod(\HH_p)\longrightarrow\Met(E_p)\,,
\end{equation}
sending $q\in\Prod(\HH_p)$ to the Hermitian metric 
$\FS(q)\in\Met(E_p)$ defined for any $s_1,\,s_2\in \HH_p$ by
\begin{equation}\label{hFSdef}
\FS(q)(s_1(x),s_2(x)):=q\left(\Pi_q(x)\,s_1,\,s_2\right)\,,
\end{equation}
where $\Pi_q(x)\in\Herm(\HH_p,q)$ is the unique orthogonal
projector with respect to $q$ satisfying
\begin{equation}\label{KerPiq}
\Ker\Pi_q(x)=\{\,s\in \HH_p~|~s(x)=0\,\}\,.
\end{equation}

\begin{defi}\label{nubaldef}
\emph{Donaldson's map} associated with $E$ and $\nu$ is defined by
\begin{equation}\label{Tnudef}
\TT_{E_p}:=\Hilb_{E_p}\circ\,\FS:\Prod(\HH_p)\longrightarrow
\Prod(\HH_p)\,.
\end{equation}
A Hermitian product $q\in\Prod(\HH_p)$ is called
\emph{$\nu$-balanced} if
\begin{equation}
\TT_{E_p}(q)=q\,.
\end{equation}
\end{defi}

Note that for any $q\in\Prod(\HH_p)$,
the metric $\FS(q)$ tautologically coincides with
the pullback $h^{FS}_q$ of the metric
induced by $q$ via the Kodaira embedding
$X\hookrightarrow\mathbb{G}(\rk(E),H^0(X,E_p)^*)$.
We thus recover the explicit
description \eqref{DonitflaEintro} for
Donaldson's map $\TT_{E_p}$.

The following Lemma, which is essentially a reformulation of
the language of \cite{Don09},
gives a first link between $\nu$-balanced
products and Berezin-Toeplitz quantization.

\begin{lem}\label{REp=1lem}
Let $q\in\Prod(\HH_p)$ be a $\nu$-balanced product,
and consider the setting of Section \ref{BTvbsec} with
$h^{E_p}:=\FS(q)$.
Then  the $L_2$-Hermitian product \eqref{L2E} satisfies
\begin{equation}\label{L2=q}
\frac{\dim\HH_p}{\Vol(X,\nu)\rk(E)}\<\cdot,\cdot\>_p=q\,,
\end{equation}
and the Rawnsley section of Definition \ref{Rawnsect} satisfies
\begin{equation}\label{REp=1}
\rho_{E_p}=\frac{\dim\HH_p}{\Vol(X,\nu)\rk(E)}\Id_{E}\,.
\end{equation}
\end{lem}
\begin{proof}
First, the identity \eqref{L2=q} is a straightforward consequence
of Definition \ref{nubaldef}. Then using \eqref{L2=q},
for any $s_1,\,s_2\in\HH_p$ and $x\in X$, we have
\begin{equation}
\begin{split}
\frac{\dim\HH_p}{\Vol(X,\nu)\rk(E)}\,q(\Pi_q(x)\,s_1,\,s_2)
&=\frac{\dim\HH_p}{\Vol(X,\nu)\rk(E)}\,\FS(q)(\ev_x^E.s_1,\ev_x^E.s_2)\\
&=\frac{\dim\HH_p}{\Vol(X,\nu)\rk(E)}\,
\left\langle\left(\ev_x^E\right)^*\ev_x^E.s_1,s_2\right\rangle_p\\
&=q\left(\left(\ev_x^E\right)^*\ev_x^E\,s_1,\,s_2\right)\,,
\end{split}
\end{equation}
which shows that
\begin{equation}\label{Piq=ev*ev}
\ev_{E_x}^*\ev_{E_x}=\frac{\dim\HH_p}{\Vol(X,\nu)\rk(E)}\Pi_q(x)\,.
\end{equation}
The identity \eqref{REp=1}
follows by applying $\ev_x^E$ on the right of \eqref{Piq=ev*ev},
using the tautological formula
$\ev_x^E\Pi_q(x)=\ev_{E_x}$ and the surjectivity of
$ev_x^E:\HH_p\fl E_{p,x}$.
\end{proof}


Note that the first equality of formula \eqref{npfla} shows that,
if the Rawnsley section is equal to a constant
scalar, then it has to be given by formula \eqref{REp=1}.
The proof of Lemma \ref{REp=1lem} then shows that if the identity
\eqref{L2=q} holds up to a multiplicative constant, then it has to hold
exactly. The normalizing
factor in front of the integral in the definition \eqref{Hilbdef}
of the Hilbert product is thus necessary for fixed points of
Donaldson's map \eqref{Tnudef} to exist.

Let $q\in\Prod(\HH_p)$ be a $\nu$-balanced product,
so that in particular \eqref{L2=q} holds, and consider the natural identifications
\begin{equation}\label{Metid}
\begin{split}
\cinf(X,\Herm(E))&\xrightarrow{~\sim~}T_{FS(q)}\Met(E_p)\\
F~&\longmapsto~\dt\Big|_{t=0}\,FS(q)(e^{tF}\cdot,\cdot)
\end{split}
\end{equation}
and
\begin{equation}\label{Prodid}
\begin{split}
\Herm(\HH_p)&\xrightarrow{~\sim~}T_q\Prod(\HH_p))\\
A~&\longmapsto~\dt\Big|_{t=0}\,\<e^{tA}\cdot,\cdot\>_p\,.
\end{split}
\end{equation}
The following result is essentially a reformulation of the language
of \cite{Don09}.

\begin{prop}\label{dFSHilb}
Let $q\in\Prod(\HH_p)$ be a $\nu$-balanced product,
and consider the setting of Section \ref{BTvbsec} with
$h^{E_p}:=\FS(q)$.
Then the differential of the
Hilbert map \eqref{HilbEfla} at $\FS(q)\in\Met(E_p)$ is given by
\begin{equation}\label{dHilbE}
\begin{split}
D_{\FS(q)}\Hilb_{E_p}:\cinf(X,\Herm(E_p))&\longrightarrow
\Herm(\HH_p)\\
F~&\longmapsto\,T_{E_p}(F)\,,
\end{split}
\end{equation}
and the differential of the
Fubini-Study map of \eqref{FSdef} at $q\in\Prod(\HH_p)$
is given by
\begin{equation}\label{dFSfla}
\begin{split}
D_q\FS:\Herm(\HH_p)&\longrightarrow
\cinf(X,\Herm(E_p))\\
A~&\longmapsto\,T^*_{E_p}(A)\,.
\end{split}
\end{equation}
\end{prop}
\begin{proof}
For any $F\in\cinf(X,\Herm(E_p))$ and $t\in\R$, set
\begin{equation}
h_t^{E_p}:=h^{E_p}(e^{tF}\cdot,\cdot)\in\Met(E_p)\,.
\end{equation}
Then for any $s_1,\,s_2\in\HH_p$ and using Proposition \ref{BTquantfla}, we have
\begin{equation}
\begin{split}
\dt\Big|_{t=0}\Hilb_{E_p}(h^{E_p}_t)(s_1,s_2)
&=\frac{\dim\HH_p}{\Vol(X,\nu)\rk(E)}\int_X\,h^{E_p}(F(x)\,s_1(x),s_2(x))\,
d\nu(x)\\
&=q\left(T_{E_p}(F)\,s_1,s_2\right)\,,
\end{split}
\end{equation}
which shows \eqref{dHilbE}.

Let us now show \eqref{dFSfla}. For any $A\in\Herm(\HH_p)$ and
$t\in\R$, set
\begin{equation}
q_t:=q(e^{tA}\cdot,\cdot)\,.
\end{equation}
Let us first show that
\begin{equation}\label{DHTpointfla}
\D{t}\Big|_{t=0}
q_t(\Pi_{q_t}(x)s_1,s_2)=q(\Pi_q(x)
A\Pi_q(x)s_1,s_2)\,.
\end{equation}
Note that for all $x\in X$, the kernel \eqref{KerPiq}
of $\Pi_{q_t}(x)$ does not depend on $t\in\R$,
so that both sides of formula
\eqref{DHTpointfla} vanish as soon as
$s_1$ or $s_2$ belong to $\Ker\Pi_q(x)$.
We thus only need to show \eqref{DHTpointfla} for
$s_1$ and $s_2$ satisfying
\begin{equation}\label{s=Ps}
s_1=\Pi_q(x)s_1\quad\text{and}\quad
s_2=\Pi_q(x)s_2\,.
\end{equation}
Taking the derivative of the projection formula
$\Pi_{q_t}(x)=\Pi_{q_t}(x)\Pi_{q_t}(x)$,
we get
\begin{equation}
\Big(\D{t}\Big|_{t=0}\Pi_{q_t}(x)\Big)
\Pi_{q_t}(x)=(\Id_{\HH_p}-\Pi_{q_t}(x))
\Big(\D{t}\Big|_{t=0}\Pi_{q_t}(x)\Big)\,.
\end{equation}
Then under the assumption \eqref{s=Ps}, we get
\begin{equation}
\begin{split}
\D{t}\Big|_{t=0}
q_t(\Pi_{q_t}&(x)s_1,s_2)=
q(A\Pi_{q}(x)s_1,s_2)+
q\Big(\Big(\D{t}\Big|_{t=0}\Pi_{q_t}(x)\Big)s_1,s_2
\Big)\\
&=q(\Pi_{q}(x)A\Pi_{q}(x)s_1,s_2)+
q\Big(\Pi_{q}(x)\Big(\D{t}\Big|_{t=0}\Pi_{q_t}(x)\Big)
\Phi_{q}(x)s_1,s_2\Big)\\
&=q(\Pi_{q}(x)A\Pi_{q}(x)s_1,s_2)\,.
\end{split}
\end{equation}
which shows \eqref{DHTpointfla}.

Now using Definition \ref{nubaldef} and Lemma \ref{REp=1lem},
the identity \eqref{DHTpointfla} implies
\begin{equation}
\begin{split}
\D{t}\Big|_{t=0}\,\FS(q_t)(s_1(x),s_2(x))
&=q(\Pi_q(x)A\Pi_q(x)s_1,s_2)\\
&=\frac{\dim\HH_p}{\Vol(X,\nu)\rk(E)}\<\ev_x^*\ev_x A\ev_x^*\ev_xs_1,s_2\>_p\\
&=\frac{\dim\HH_p}{\Vol(X,\nu)\rk(E)}\FS(q)(\ev_x A\ev_x^* s_1(x),s_2(x))\,,
\end{split}
\end{equation}
so that to establish \eqref{dFSfla}, it suffices to show
\begin{equation}\label{evAev=T*A}
\frac{\dim\HH_p}{\Vol(X,\nu)\rk(E)}\ev_x A\ev_x^*=T^*(A)(x)\,.
\end{equation}
Using Definition \ref{BTquantdefE}, for any $F\in\cinf(X,\Herm(E_p))$
we compute
\begin{equation}
\begin{split}
\Tr^{\HH_p}[A\,T_p(F)]&=\int_X\Tr^{\HH_p}
[A\,\ev_x F(x)\ev_x^*]\,d\nu(x)\\
&=\int_X\Tr^{E_p}[\ev_x^*A\ev_x F(x)]\,d\nu(x)\,.
\end{split}
\end{equation}
By definition of the dual with
respect to the $L_2$-scalar product \eqref{L2alphaE} and
using Lemma \ref{REp=1lem} again, this shows \eqref{evAev=T*A}.
\end{proof}

\begin{cor}\label{dTTE}
Let $q\in\Prod(\HH_p)$ be a $\nu$-balanced product,
and consider the setting of Section \ref{BTvbsec} with
$h^{E_p}:=\FS(q)$.
Then the differential of Donaldson's map $\TT_{E_p}$
at $q\in\Prod(\HH_p)$ satisfies
\begin{equation}\label{DTT=TT*E}
D_{q}\TT_{E_p}=T_{E_p}\circ T^*_{E_p}:\Herm(\HH_p)\longrightarrow\Herm(\HH_p)\,.
\end{equation}
In particular, its positive eigenvalues with multiplicity
as an endomorphism of $\Herm(\HH_p)$ coincide
with the positive
eigenvalues of the Berezin transform associated with
$h^{E_p}=\FS(q)$.
\end{cor}
\begin{proof}
The identity \eqref{DTT=TT*E} is an immediate consequence of
Proposition \ref{dFSHilb}. For the second statement, note
that $T^*_{E_p}$ maps isomorphically
the eigenspace associated with a non-zero eigenvalue of
$T_{E_p}^*\circ T_{E_p}$ acting on $\Herm(\HH_p)$
to the eigenspace associated with the same eigenvalue of
$\cB_p:=T_{E_p}\circ T_{E_p}^*$ acting on $\cinf(X,\Herm(E_p))$.
\end{proof}

%

Recall from \cite{Don87}
and \cite{UY86} that $E$ is
\emph{Mumford stable}
if and only if $E$ is simple
and admits a Hermitian metric $h^E$
satisfying the Hermite-Einstein equation
\eqref{weakHE}.
For any Hermitian metric $h^{E_p}\in\Met(E_p)$, write
$h^{L^{-p}}\otimes h^{E_p}\in\Met(E)$ for the
Hermitian metric on $E$ induced by $h^{L}$ and $h^{E}$.
We then have the following straightforward reformulation of a
fundamental result of Wang in \cite{Wan05}.

\begin{theorem}\label{balvb}
\cite[Th.\,1.2]{Wan05}
Assume that $E$ is Mumford stable and that $\nu$ is
the Liouville measure.
Then there is $p_0\in\N$ such that for any $p\geq p_0$, there exists
a $\nu$-balanced
product $q_p\in\Prod(\HH_p)$, unique up to multiplicative
constant. Furthermore, there is an explicit constant $c_p>0$
for all $p\geq p_0$
such that the following
smooth convergence holds,
\begin{equation}\label{balvbfla}
c_p\,h^{L^{-p}}\otimes\FS(q_p)\xrightarrow{p\fl+\infty}e^f h^E\,,
\end{equation}
where $h^E\in\Met(E)$
satisfies the Hermite-Einstein equation
\eqref{weakHE} and $f\in\cinf(X,\R)$ satisfies
$\Delta\,f=2\pi\big(\scal(\om)-\int_X\scal(\om)\,\frac{d\nu}{\Vol(X)}\big)$.
\end{theorem}

Note that the limit metric in \eqref{balvbfla} is precisely what is called
a \emph{weak Hermite-Einstein metric} by Wang in \cite[Th.\,1.2]{Wan05},
which can redily be seen by comparing
the Hermite-Einstein equation \eqref{weakHE} with his weak Hermite-Einstein
equation \cite[Th.\,1.2, (2)]{Wan05} and using the last equality of
\eqref{omddbar=delta}.

The following result on the convergence of Donaldson's iterations has
then been established by Seyyedali in \cite{Sey09}.

\begin{theorem}\label{TTEcvth}
\cite{Sey09}
Aussume that $E$ is Mumford stable and that $\nu$ is
the Liouville measure. Then there is $p_0\in\N$
such that for any $p\geq p_0$ and any $q\in\Prod(\HH_p)$,
there exists a $\nu$-balanced Hermitian product
$q_p\in\Prod(\HH_p)$ such that
\begin{equation}\label{TTEcvfla}
\TT_{E_p}^r(q)\xrightarrow{r\fl+\infty}q_p\,.
\end{equation}
\end{theorem}

%
%

Let us now take for
$(E,h^E)$ the trivial Hermitian line bundle.
For any holomorphic line bundle $L$,
write $\Met^+(L)\subset\Met(L)$ for the space of
\emph{positive} Hermitian metrics on $L$, which are the Hermitian metrics
$h\in\Met(L)$ such that the associated $2$-form
$\om_h\in\Om^2(X,\R)$ defined by formula \eqref{preq} is Kähler,
so that the bilinear product $g_h^{TX}$ defined by formula \eqref{gTX}
is a Riemannian metric.
We write $\Vol_h(X)>0$ for the Riemannian volume of $(X,g^{TX}_h)$
and $\Delta_h$
for the associated Laplace-Beltrami operator acting on $\cinf(X,\C)$.
The \emph{scalar Hilbert map} is the map
\begin{equation}\label{Hilbdef}
\begin{split}
\Hilb_p:\Met^+(L^p)&\longrightarrow\Prod(\HH_p)\\
h_p\,&\longmapsto\frac{\dim\HH_p}{\Vol_{h_p}(X)}\int_X\,h_p(\cdot,\cdot)\,
\frac{\om_{h_p}^n}{n!}\,.
\end{split}
\end{equation}

Using \emph{Kodaira's embedding theorem}, let us take
$p\in\N$ big enough so that the
Fubini-Study map \eqref{FSdef} associated with $L^p$
takes values in the space $\Met^+(L^p)$ of
\emph{positive} Hermitian metrics.
We then have the following analogue of Definition \ref{nubaldef}.

\begin{defi}\label{baldef}
\emph{Donaldson's map} associated with $(X,L)$ is defined
by
\begin{equation}\label{Tdef}
\TT_p:=\Hilb_p\circ\,\FS:\Prod(\HH_p)\longrightarrow
\Prod(\HH_p)\,.
\end{equation}
A Hermitian product $q\in\Prod(\HH_p)$ is called
\emph{balanced} if 
\begin{equation}
\TT_p(q)=q\,,
\end{equation}
that is, if $q\in\Prod(\HH_p)$
is $\frac{\om_{\FS(q)}^n}{n!}$-balanced in the
sense of Definition \ref{nubaldef}.
\end{defi}

As in Section \ref{vbcase}, we recover the explicit
description \eqref{Tdefcan} for
Donaldson's map $\TT_p$.

Let $q\in\Prod(\HH_p)$ be balanced,
and recall the natural identifications
\eqref{Metid} and \eqref{Prodid} with $E=\C$.

\begin{prop}\label{dHilb}
Let $q\in\Prod(\HH_p)$ be a balanced product,
and consider the setting of Section \ref{BTPOVMsec} with $h^p:=\FS(q)$
and $d\nu_X:=\om_{\FS(q)}^n/n!$.
Then the differential of the
Hilbert map \eqref{Hilbdef} at $\FS(q)\in\Met(L^p)$ is given by
\begin{equation}\label{dHilbfla}
\begin{split}
D_{\FS(q)}\Hilb_p:\cinf(X,\R)&\longrightarrow
\Herm(\HH_p)\\
f\,&\longmapsto\,T_p\left(f+
\frac{1}{4\pi}\Delta_{\FS(q)} f\right)\,,
\end{split}
\end{equation}
where $\Delta_{\FS(q)}$
is the Laplace-Beltrami operator of $(X,g_{\FS(q)}^{TX})$
acting on $\cinf(X,\C)$.
\end{prop}
\begin{proof}
For any $f\in\cinf(X,\R)$ and $t\in\R$, set
\begin{equation}
h_t:=e^{tf}\,h^p\in\Met(L^p)\,.
\end{equation}
We will use the following classical formula of Kähler geometry,
\begin{equation}\label{dtvol}
\dt\Big|_{t=0}\,\frac{\om_{h_t}^n}{n!}=\frac{1}{4\pi}\Delta_{h^p} f~\frac{\om_{h^p}^n}{n!}\,.
\end{equation}
Then for any $s_1,\,s_2\in\HH_p$ and by Proposition \ref{BTquantfla}, we have
\begin{equation}
\begin{split}
&\dt\Big|_{t=0}\Hilb_p(h_t)(s_1,s_2)\\
&=\frac{\dim\HH_p}{\Vol_{h^p}(X)}\left(\int_X\,\dt\Big|_{t=0}
h_t(s_1,s_2)\,
dv_{h^p}+\int_X\,h^p(s_1,s_2)\,\dt\Big|_{t=0}dv_{h_t}\right)\\
&=\frac{\dim\HH_p}{\Vol_{h^p}(X)}\int_X
\left(f+\frac{1}{4\pi}\Delta_{h^p} f
\right)h^p(s_1,s_2)\,\frac{\om_{h^p}^n}{n!}\\
&=q\left(
T_p\left(f+\frac{1}{4\pi}\Delta_{h^p} f\right)\,s_1,s_2\right)\,,
\end{split}
\end{equation}
which proves \eqref{dHilbfla} taking $h^p=\FS(q)$.
\end{proof}

\begin{cor}\label{dTT}
Let $q\in\Prod(\HH_p)$ be a balanced product,
and consider the setting of Section \ref{BTPOVMsec} with $h^p:=\FS(q)$
and $d\nu_X:=\om_{\FS(q)}^n/n!$.
Then the differential of Donaldson's map $\TT_p$
at $q\in\Prod(\HH_p)$ satisfies
\begin{equation}\label{DTT=TT*}
D_{q}\TT_p=T_p
\left(1+\frac{1}{4\pi}\Delta_{\FS(q)}\right)T^*_p\,.
\end{equation}
\end{cor}
\begin{proof}
This is a consequence of the Definition \ref{baldef} for $\TT_p$,
together with Proposition \ref{dFSHilb} on the differential of the
Fubini-Study map
associated with $L^p$ and Proposition \ref{dHilb} on the differential
of the
scalar Hilbert map.
\end{proof}
%

Let $\Aut(X,L)$ denote
the group of holomorphic automorphisms of $X$ lifting holomorphically
to $L$. The relevance of the balanced products of
Definition \ref{baldef} in Kähler geometry is illustrated
by the following celebrated result of Donaldson.

\begin{theorem}\label{balex}
\cite{Don01}
Assume that $\Aut(X,L)$
is discrete and that there exists a positive Hermitian metric
$h_\infty\in\Met^+(L)$
such that the induced Kähler
metric $g_{\infty}^{TX}$ has \emph{constant scalar curvature}.
Then there is $p_0\in\N$ such that for
any $p\geq p_0$,
there exists a unique balanced metric $g^{TX}_p$ associated with $L^p$.
Furthermore, the following smooth convergence holds,
\begin{equation}\label{balcv}
\frac{1}{p}\,g^{TX}_p\xrightarrow{p\fl+\infty}
g_{\infty}^{TX}\,.
\end{equation}
\end{theorem}

On the other hand, the following result on the convergence of
the iterations of Donaldson's map \eqref{Tdef} has been
established by Donaldson in \cite{Don05} and Sano in \cite{San06}.

\begin{theorem}\label{TTpcvth}
\cite{Don05},\cite{San06}
Under the assumptions of Theorem \ref{balex},
there exist $p_0\in\N$ such that for
any $p\geq p_0$ and any $q\in\Prod(\HH_p)$, there exists
a balanced product $q_p\in\Prod(\HH_p)$ such that
\begin{equation}\label{TTpcvfla}
\TT^r_p(q)\xrightarrow{r\fl+\infty}q_p\,.
\end{equation}
\end{theorem}



\subsection{Scalar case and proof of Theorem \ref{gapvb}}\label{scalspecsec}

Recall that in the case $E=\C$, the Laplacian \eqref{KodLapdef} is the
Laplace-Beltrami operator $\Delta_h$ of $(X,g_h^{TX})$
acting on $\cinf(X,\C)$.
Recall the $L^2$-Hermitian product $\<\cdot,\cdot\>_{L_2}$
on $\cinf(X,\C)$
induced by the Riemannian measure of $g_h^{TX}$, and let us write $\|\cdot\|_{H^m}$ for the
Sobolev norm \eqref{ellest} of order $m\in\N$.
Using the functional calculus \eqref{calculfct},
we define an operator on
$L_2(X,\C)$ by the formula
\begin{equation}\label{Spdef}
S_p:=\left(1+\frac{\Delta_h}{4\pi p}
\right)^{1/2}\cB_p
\left(1+\frac{\Delta_h}{4\pi p}\right)^{1/2}\,,
\end{equation}
where $\cB_p$ is the Berezin transform of Definition \ref{BpdefE}
for $E=\C$. This definition is motivated by the following Lemma.

\begin{lem}\label{SpecTTp}
Let $q\in\Prod(\HH_p)$ be a balanced Hermitian
product,
and let $h\in\Met^+(L)$ be defined by $h^p:=h_{\FS(q)}$.
Then $S_p$ is a smoothing self-adjoint operator with respect
to $\<\cdot,\cdot\>_{L_2}$,
and its positive eigenvalues coincide with
the positive eigenvalues of Donaldson's map $\TT_p$
of Definition \ref{baldef}.
\end{lem}
\begin{proof}
Under the assumptions of the statement, Lemma \ref{REp=1lem} shows that
the $L_2$-products \eqref{L2} and \eqref{L2alphaE} for $E=\C$
satisfy
\begin{equation}
\<\cdot,\cdot\>_{W_p}=\frac{\dim\HH_p}{\Vol_h(X)}\<\cdot,\cdot\>_{L_2}\,.
\end{equation}
In particular, as $\cB_p$ is smoothing and self-adjoint
with respect to $\<\cdot,\cdot\>_{W_p}$, it is also self-adjoint
with respect to $\<\cdot,\cdot\>_{L_2}$, so that
$S_p$ is itself a smoothing and self-adjoint operator with respect
to $\<\cdot,\cdot\>_{L_2}$.

Recall on the other hand that as $h^p=\FS(q)$, we have
$p\,\Delta_{\FS(q)}=\Delta_h$.
Using Corollary \ref{dTT}, this implies that $D_{q}\TT_p$ is a self-adjoint
operator satisfying
\begin{equation}\label{dT=TdT}
\begin{split}
D_{q}\TT_p&=T_p\left(1+\frac{\Delta_h}{4\pi p}\right)T_p^*\\
&=
\left(T_p\left(1+\frac{\Delta_h}{4\pi p}\right)^{1/2}\right)
\left(T_p\left(1+\frac{\Delta_h}{4\pi p}\right)^{1/2}\right)^*\,.
\end{split}
\end{equation}
Following the argument of the proof of Corollary \ref{dTT},
we then see that the map $\left(T_p\left(1+\frac{\Delta_h}{4\pi p}\right)^{1/2}\right)^*$ maps isomorphically the eigenspace
associated with a non-zero eigenvalues of $D_{q}\TT_p$ acting on
$\Herm(\HH_p)$ to the eigenspace associated with the same
eigenvalue of $S_p$ acting on $\cinf(X,\C)$, which gives the result.
\end{proof}

Let us now consider a general $h\in\Met^+(L)$,
and let us study the behavior of the
operator $S_p$ of \eqref{Spdef} as $p\fl+\infty$.
Let us denote $D_h$ for
the operator \eqref{varscalopdef} acting on $\cinf(X,\R)$
describing the
variation of the scalar curvature of $g_h^{TX}$.

\begin{prop}\label{KS'}
For any $m\in\N$, there exists $C_m>0$ and $l\in\N$ such that for any
$f\in\cinf(X,\C)$ and all $p\in\N$, we have
\begin{equation}\label{KSexp'}
\left\|S_p f-
\left(1-\frac{D_h}{8\pi p^2}
\right)f\right\|_{H^m}\leq
\frac{C_m}{p^3}\|f\|_{H^{m+8}}\;,
\end{equation}
uniformly in the $\CC^m$-norm of the derivatives of $h$
up to order $l$.
\end{prop}
\begin{proof}
Following the proof of Proposition
\ref{KS_KE}, we know that for all $m\in\N$,
there is $C_m>0$ such that
\begin{equation}\label{KSexpab'}
\left|\cB_pf-\left(1+\frac{D_2}{p}+\frac{D_4}{p^2}\right) f
\right|_{\CC^m}\leq
\frac{C_m}{p^3}|f|_{\CC^{m+6}}\;,
\end{equation}
where $D_2$ and $D_4$ are the second order and
fourth order differential operators given by formulas
\eqref{D2} and \eqref{D4pre}.
Now following e.g. \cite[(5)]{Fin12} with our conventions,
we have the following formula for the variation of
scalar curvature operator,
\begin{equation}
D_h\,f=\frac{\Delta_h^2}{4\pi}f+\frac{\sqrt{-1}}{\pi}
\<\Ric(\om),\partial\dbar f\>_{g_h}\,,
\end{equation}
so that we get
\begin{equation}\label{D4}
D_4=\frac{\Delta_h^2}{16\pi^2}-\frac{D_h}{8\pi}\,.
\end{equation}
On the other hand, $\Delta_h$ commutes
with $\left(1+\frac{\Delta_h}{4\pi p}\right)^{1/2}$, and
by definition \eqref{ellest}
of the Sobolev norms,
there exists a constant $C_0>0$
such that, for all $f\in\cinf(X,\R)$ and $p\in\N$,
we have $\|f-\left(1+\frac{\Delta_h}{4\pi p}\right)^{1/2}f\|_{H^m}
\leq C_0p^{-1}\|f\|_{H^{m+2}}$.
Using formulas \eqref{D2} and \eqref{D4}, we thus
get a constant $C>0$ sucht for any $f\in\cinf(X,\R)$
and $p\in\N$,
\begin{equation}\label{details}
\begin{split}
&\left\|\left(1+\frac{\Delta_h}{4\pi p}\right)^{1/2}\left(1+\frac{D_2}{p}+\frac{D_4}{p^2}\right) \left(1+\frac{\Delta_h}{4\pi p}\right)^{1/2}f\right\|_{H^m}\\
&=\left\|f+\frac{\Delta_h^3}{64\pi^3p^3}f-\left(1+\frac{\Delta_h}{4\pi p}\right)^{1/2}
\frac{D_h}{8\pi p^2}\left(1+\frac{\Delta_h}{4\pi p}\right)^{1/2}f\right\|_{H^m}\\
&\leq\left\|f-\frac{D_h}{8\pi p^2}f\right\|_{H^m}
+\frac{1}{p^3}\left\|\frac{\Delta_h^3}{64\pi^3}f\right\|_{H^m}\\
&+
\frac{1}{p^2}\left\|\frac{D_h}{8\pi}f-\left(1+\frac{\Delta_h}{4\pi p}\right)^{1/2}
\frac{D_h}{8\pi}\left(1+\frac{\Delta_h}{4\pi p}\right)^{1/2}f\right\|_{H^m}\\
&\leq\left\|f-\frac{D_h}{8\pi p^2}f\right\|_{H^m}+\frac{C}{p^3}
\|f\|_{H^{m+8}}\,.
\end{split}
\end{equation}
Taking the expansion \eqref{KSexpab'} into definition \eqref{Spdef}
of $S_p$, formula \eqref{KSexp'} for $m=0$ follows from the estimate \eqref{details}.
The case of general $m\in\N$ follows in the same way
using formula \eqref{ellest} for the Sobolev norm.
\end{proof}

\begin{prop}\label{boundexp'}
For any $m,\,k_1,\,k_2\in\N$, there exists $C>0$ and $l\in\N$
such that for any
$f\in\cinf(X,\C)$ and all $p\in\N$, we have
\begin{equation}\label{boundexpfla'}
\left\|\left(\frac{\Delta_h}{p}\right)^{k_1}
\left(e^{-\frac{\Delta_h}{4\pi p}}
-\cB_p\right)
\left(\frac{\Delta_h}{p}\right)^{k_2}f\,\right\|_{H^m}\leq
\frac{C}{p}\|f\|_{H^m}\;,
\end{equation}
uniformly in the
$\CC^m$-norm of the derivatives of
$h$ up to order $l$.
\end{prop}
\begin{proof}
The case $m=0$ follows from the uniformity in the
estimates of the Bergman kernel of \cite[Th.\,4.2.1]{MM07}
and the analogous result of Lu and Ma
in the appendix of \cite[Th.\,25]{Fin10} for the $Q_K$-operator.
This proof readily extends to the case of general
$m\in\N$, following the analogous extension in the proof of
\cite[Prop.\,3.9]{IKPS19}.
\end{proof}

\begin{cor}\label{boundexpcor}
For any $m\in\N$, there exists $C>0$ and $l\in\N$ such that for any
$f\in\cinf(X,\C)$ and all $p\in\N$, we have
\begin{equation}\label{boundexpflacor}
\left\|\left(\Big(1+\frac{\Delta_h}{4\pi p}
\Big)e^{-\frac{\Delta_h}{4\pi p}}
-S_p\right)f\right\|_{H^m}\leq
\frac{C_m}{p}\|f\|_{H^m}\;,
\end{equation}
uniformly in the
$\CC^m$-norm of the derivatives of
$h$ up to order $l$.
\end{cor}
\begin{proof}
Fix $p\in\N$, and consider
the functional calculus \eqref{calculfct}
with $\Psi:\R\fl\R$ given by $\Psi(s)=(1+s/4\pi p)^{-1/2}$
for $s\geq 0$ and $\Psi(s)=0$ otherwise,
so that in particular $\Psi(s)\leq 1$ for all $s\in\R$.
Then using the elliptic estimates \eqref{ellest} and the
fact that $\Delta_h$ commutes with $\Psi(\Delta_h)$,
for any
$f\in\cinf(X,\C)$ and $m\in\N$, we get
\begin{equation}
\left\|\Big(1+\frac{\Delta_h}{4\pi p}\Big)^{-1/2}f\right\|_{H^m}
\leq\|f\|_{H^m}\,.
\end{equation}
Now by definition \eqref{Spdef} of the operator
$S_p$, using Proposition \ref{boundexp} and the fact that any
function of $\Delta_h$ commutes with the heat operator
$e^{-\frac{\Delta_h}{4\pi p}}$, this implies that for all
$m\in\N$, there exists $C_m>0$ such that for all
$f\in\cinf(X,\C)$ and $p\in\N$,
\begin{equation}
\begin{split}
\Big\|\Big(\Big(1+\frac{\Delta_h}{4\pi p}
\Big)&e^{-\frac{\Delta_h}{4\pi p}}
-S_p\Big)f\Big\|_{H^m}\\
&=
\left\|\Big(1+\frac{\Delta_h}{4\pi p}
\Big)^{1/2}\left(e^{-\frac{\Delta_h}{4\pi p}}
-\cB_p\right)\Big(1+\frac{\Delta_h}{4\pi p}
\Big)^{1/2}f\right\|_{H^m}\\
&\leq
\left\|\Big(1+\frac{\Delta_h}{4\pi p}
\Big)\left(e^{-\frac{\Delta_h}{4\pi p}}
-\cB_p\right)\Big(1+\frac{\Delta_h}{4\pi p}
\Big)^{1/2}f\right\|_{H^m}\\
&\leq\frac{C_m}{p}\left\|\Big(1+\frac{\Delta_h}{4\pi p}
\Big)^{-1/2}f\right\|_{H^m}\leq\frac{C_m}{p}\left\|f\right\|_{H^m}\,.
\end{split}
\end{equation}
This proves the result.

\end{proof}

\begin{prop}\label{hjrefp'}
For any $L>0$ and $m\in\N$, there exist constants $C_{m}>0$ and
$p_m\in\N$,
uniform in the $\CC^m$-norm of the derivatives of
$h$ up to some finite order,
such that for any $p\geq p_m,\,\mu\in\C$ and $f\in\cinf(X,\C)$
satisfying
\begin{equation}\label{Sp=muf}
S_pf=\mu f~~\text{and}~~p|1-\mu|<L\,,
\end{equation}
we have
\begin{equation}\label{hjrefined}
\|f\|_{H^m}\leq C_{m}\|f\|_{L_2}\;.
\end{equation}
\end{prop}
\begin{proof}
For any $p\in\N$, $f\in\cinf(X,\C)$ and $\mu\in\C$
such that $S_pf=\mu f$, we have
\begin{equation}\label{deltFest}
\begin{split}
p\Big(\Big(1+\frac{\Delta_h}{4\pi p}\Big)
e^{-\frac{\Delta_h}{4\pi p}}
-S_p\Big)f&=p(1-\mu)f-
p\Big(1-\Big(1+\frac{\Delta_h}{4\pi p}\Big)e^{-\frac{\Delta_h}
{4\pi p}}\Big)f\\
&=p(1-\mu)f-\frac{\Delta_h}
{4\pi}\Psi(\Delta_h/4\pi p)\,f\;,
\end{split}
\end{equation}
where the bounded operator $\Psi(\Delta_h/4\pi p)$ acting on
$L_2(X,\C)$ is defined as in \eqref{calculfct} for the
continuous function $\Psi:\R\fl\R$ given for any $s\in\R^*$ by
\begin{equation}
\Psi(s):=\frac{1-e^{-s}(1+s)}{s}\,.
\end{equation}
Using Corollary \ref{boundexpcor} and formula \eqref{ellest},
formula \eqref{deltFest} implies that for any $L>0$ and $m\in\N$,
there is a constant $C>0$, uniform
in the $\CC^m$-norm of the derivatives of
$h$ up to some finite order, such that
for any $f\in\cinf(X,\C)$ and $\mu\in\C$ satisfying \eqref{Sp=muf},
we have
\begin{equation}\label{sobfest}
\|\Psi(\Delta_h/4\pi p)f\|_{H^{m+2}}\leq C\|f\|_{H^m}\;.
\end{equation}
On the other hand, using Proposition \ref{boundexp'} again, we get
\begin{equation}\label{hjdeltFest}
\begin{split}
\|\Psi(\Delta_h/4\pi p)f\|_{H^m}&\geq
\Big\|\Psi(\Delta_h/4\pi p)f+\Big(S_p-\Big(1+\frac{\Delta_h}
{4\pi p}\Big)
e^{-\frac{\Delta_h}{4\pi p}}\Big)f\Big\|_{H^m}\\
&\quad\quad-\Big\|\Big(S_p-\Big(1+\frac{\Delta_h}{4\pi p}\Big)e^{-\frac{\Delta_h}{4\pi p}}\Big)f\Big\|_{H^m}\\
&\geq\inf_{s>0}\,\{\Psi(s)+\mu-(1+s)e^{-s}\}\,\|f\|_{H^m}
-C_m p^{-1}\|f\|_{H^m}\;.
\end{split}
\end{equation}
Combining \eqref{sobfest} and \eqref{hjdeltFest},
we see that for any $L>0$ and $m\in\N$, there exists
$p_0\in\N$ and $C>0$, uniform
in the $\CC^m$-norm of the derivatives of
$h$ up to some finite order, such that for any $p\geq p_0,\,
f\in\cinf(X,\C)$ and $\mu\in\C$ satisfying \eqref{Sp=muf},
we have
\begin{equation}
\|f\|_{H^{m+2}}\geq C\|f\|_{H^m}\,.
\end{equation}
This implies formula \eqref{hjrefined} by induction on $m\in\N$.
\end{proof}

Let now $h_\infty\in\Met^+(L)$ be such that $g_{h_\infty}^{TX}$ has
constant scalar curvature, and recall that
the associated variation of scalar curvature operator
$D_{h_\infty}$ of \eqref{varscalopdef} is then a positive elliptic
self-adjoint operator. Write
\begin{equation}\label{evmu}
0=\mu_0\leq\mu_1
\leq\dots\leq\mu_j\leq\dots
\end{equation}
for the increasing sequence of its eigenvalues.

On the other hand, Theorem \ref{balex} gives us
Hermitian metrics $h_p\in\Met^+(L)$ for all $p\in\N$ big enough,
such that the K\"ahler metrics $g_{h_p^p}^{TX}$ associated 
with $h_p^p\in\Met^+(L^p)$ are balanced
and such that $h_p\to h_{\infty}$ as $p\to+\infty$.
Then by Lemma \ref{SpecTTp},
the smoothing operator $S_p$ defined in \eqref{Spdef} for
$h_p\in\Met^+(L)$ is self-adjoint
with respect to $\<\cdot,\cdot\>_{L_2}$.
Write
\begin{equation}\label{gammaomp}
\mu_{0,p}\geq\mu_{1,p}
\geq\dots\geq\mu_{j,p}\geq\dots\geq 0
\end{equation}
for the decreasing sequence of its eigenvalues.

\begin{theorem}\label{estspecSpth}
Assume that $h_p\in\Met^+(L)$ is such that $g_{h_p^p}^{TX}$
is balanced, for all $p\in\N$ big enough.
Then the eigenvalues \eqref{gammaomp} of the associated operator
$S_p$ as in \eqref{Spdef} satisfy the following estimate as
$p\fl+\infty$,
\begin{equation}
1-\mu_{j,p}=\frac{\mu_j}{8\pi p^2}+o(p^{-2})\,.
\end{equation}
\end{theorem}
\begin{proof}
By Theorem \ref{balex} and using the uniformity in Proposition \ref{KS'},
the Sobolev embedding theorem
implies that there exists a sequence $\epsilon_p\to 0$ as $p\fl+\infty$
such that
for any $f\in\cinf(X,\C)$, we have
\begin{equation}\label{KSSob'}
\left\|p(1-S_p)f-\frac{D_{h_\infty}}{8\pi p}f\right\|_{L_2}
<\epsilon_p\,\|f\|_{H^{m}}\;,
\end{equation}
for some $m\in\N$ large enough.
For any $j\in\N$, write $e(\mu_j)\in\cinf(X,\C)$
for the normalized eigenfunction associated with
the $j$-th eigenvalue of the operator $D_{h_\infty}$ as in \eqref{evmu}.
Then for all $p\in\N$, we have
\begin{equation}\label{estevdelt'}
\left\|p(1-S_p)e(\mu_j)-\frac{D_{h_\infty}}{8\pi p}e(\mu_j)\right\|
_{L_2}<\epsilon_p\,.
\end{equation}
Conversely, fix $L>0$ and consider a sequence of normalized eigenfunctions
$\{f_p\}_{p\in\N}$ of the sequence of
operators $\{S_p\}_{p\in\N}$ associated with eigenvalues
$\{\mu_p\}_{p\in\N}$ satisfying $p|1-\mu_p|<L$.
Using the uniformity
in Proposition \ref{hjrefp'} and by
the inequality \eqref{KSSob'}, we get a sequence $\epsilon_p\to 0$ as
$p\fl+\infty$
such that
\begin{equation}\label{specinv>'}
\left\|p(1-\mu_p)f_p-\frac{D_{h_\infty}}{8\pi p}
f_p\right\|_{L_2}<\epsilon_p\;.
\end{equation}
%
The rest of the proof is then strictly analogous to the proof
of Theorem \ref{mainthvb} at the end of Section \ref{BTKEsec}.
\end{proof}

\medskip\noindent{\bf Proof of Theorem \ref{gapvb}.}
Let us first deal with Statement 2, since Statement 1
is essentially a vector bundle version of the same argument.
By \cite[Def.\,4.3,\,Lemma\,4.4]{Sze14},
if there does not exist any holomorphic
vector fields over $X$, the kernel of $D_{h_\infty}$
is generated by the constant
function. As $\Aut(X,L)$ is discrete, there is no
holomorphic vector fields over $X$, and its two
first eigenvalues in \eqref{evmu} satisfy
\begin{equation}\label{mu1>0}
\mu_0=0~~\text{and}~~\mu_1>0\,.
\end{equation}
Combining Lemma \ref{SpecTTp} with Theorem \ref{estspecSpth},
we then get that the differential $D_{q_p}\TT_p$ of Donaldson's map
\eqref{Tdef} at a balanced product $q_p\in\Prod(\HH_p)$,
which satisfies $D_{q_p}\TT_{p}(\Id_{\HH_p})=\Id_{\HH_p}$ by definition,
has a sequence of decreasing eigenvalues 
$\{\beta_{k,p}\}_{k\in\N}$ such that as $p\to+\infty$,
\begin{equation}\label{betapprebal}
\beta_{0,p}=1\quad\text{and}\quad
\beta_{1,p}=1-\frac{\mu_1}{8\pi p^2}+o(p^{-2})\,.
\end{equation}
In particular, equation \eqref{mu1>0} implies that $\beta_{k,p}<1$
for all $k\geq 1$ and all $p\in\N$ big enough.
Finally, following for instance \cite[Prop.\,4.8]{IKPS19},
we know that the dual map $T^*_p$ is injective
for all $p\in\N$ big enough, while the operator
$\left(1+\frac{\Delta_p}{4\pi p}\right)^{1/2}$ is strictly positive,
hence injective. Thus by Corollary \ref{dTT}, we see that
$D_{q_p}\TT_p$ is injective as well, for all $p\in\N$ big enough.
Hence for any such $p\in\N$,
we can apply the classical Grobman-Hartman theorem
as in \cite[\S\,4]{IKPS19} to find coordinates around
$q_p$ in $\Prod(\HH_p)$ in which $\TT_{p}$ coincide with
its linearization $D_{q_p}\TT_{p}$. Using Theorem \ref{TTEcvth}
and formula \eqref{betapprebal}, we then get the exponential convergence
\eqref{expcvest}, with rate $\beta_p:=\beta_{1,p}$.

To establish Statement 1,
let $h^E$ satisfy the Hermite Einstein equation \eqref{weakHE},
let $f\in\cinf(X,\R)$ satisfy
$\Delta\,f=2\pi\big(\scal(\om)-\int_X\scal(\om)\,\frac{d\nu}{\Vol(X)}\big)$,
and write $\nabla^E$ for the
Chern connection on $(E,e^f h^E)$.
Recall that for any $F\in\cinf(X,\End(E))$,
the induced Chern connection $\nabla^{\End(E)}$ on $\End(E)$
satisfies
\begin{equation}
\nabla^{\End(E)}F=[\nabla^E,F]\,.
\end{equation}
Now if $F\in\cinf(X,\End(E))$ satisfies
$[\nabla^E,F]=0$, then its characteristic subbundles are
holomorphic subbundles of $E$, and as $E$ is simple, this
implies that $F=c\,\Id_E$ for some $c\in\C$.
In particular,
the kernel of the Bochner Laplacian \eqref{delta}
is $1$-dimensional, generated by $\Id_E\in\cinf(X,\End(E))$.
On the other hand, note that as $h^E$ satisfies \eqref{weakHE} and
$f\in\cinf(X,\R)$ satisfies
$\Delta\,f=2\pi\big(\scal(\om)-\int_X\scal(\om)\frac{d\nu}{\Vol(X)}\big)$,
writing $R_{e^fh^E}\in\Om^2(X,\End(E))$ for the Chern curvature
of $(E,e^f h^E)$, we get
\begin{equation}\label{trueweakHE}
\frac{\sqrt{-1}}{2\pi}\<\om,R_{e^fh^E}\>_{g^{TX}}=
\left(c+\int_X\frac{\scal(\om)}{2}\frac{d\nu}{\Vol(X)}-\frac{\scal(\om)}{2}\right)\,\Id_E\,.
\end{equation}
Proposition \ref{Weitzenfla} then shows that
the associated Bochner Laplacian \eqref{delta} coincides with twice the
associated Kodaira Laplacian \eqref{KodLapdef}, and its first eigenvalues
in \eqref{lapev}
satisfy $\lambda_0^E=0$ and $\lambda_1^E>0$.
The rest of the proof of Statement 1 then follows the same argument as the proof
of Statement 2 above.

\qed

\subsection{Moment map for balanced embeddings}
\label{momsec}

In this section, we study the behaviour of the moment maps
associated with \emph{balanced embeddings}, introduced by
Donaldson \cite{Don01} and Wang \cite{Wan05} in their study of canonical
metrics in complex geometry. We relate the derivative of these moment maps
at a balanced embedding with the derivative of Donaldson's maps,
and show how Theorem \ref{gapvb} gives a lower bound
on their spectral gap, recovering the results of \cite[Th.\,5]{KMS16}
and \cite[Th.\,7]{Fin12}.

Consider first the setting of Section \ref{vbcase}, fix $p\in\N$ big enough
and let $q\in\Prod(\HH_p)$ be a $\nu$-balanced Hermitian
product in the sense of Definition \ref{nubaldef}.
Let $\mathbb{G}(\rk(E),\HH_p)$ be the Grassmanian
of rank-$\rk(E)$ planes inside $\HH_p$, and 
let us consider the natural injection
\begin{equation}\label{rankE}
\begin{split}
\mathbb{G}(\rk(E),\HH_p)&\longhookrightarrow\Herm(\HH_p)\\
z &\longmapsto\Pi_z\,,
\end{split}
\end{equation}
sending a rank-$\rk(E)$ plane $z\in\mathbb{G}(\rk(E),\HH_p)$
to the orthogonal projection $\Pi_z\in\Herm(\HH_p)$ on this plane.
By Kodaira's embedding theorem, there exists $p_0\in\N$
such that for all $p\geq p_0$,
the evaluation map \eqref{evE} is surjective for all $x\in X$,
and induces a natural embedding
\begin{equation}\label{Kodvbdef}
\Kod_p^E:X\longrightarrow\mathbb{G}(\rk(E),\HH_p))\,.
\end{equation}
%
%
Let us write $\GL(\HH_p)$ for the group of invertible
endomorphisms of $\HH_p$, and $U(\HH_p)$ for its unitary group.
The following definition is a reformulation of the \emph{moment map}
used in \cite{Wan05}.

\begin{defi}
The \emph{moment map for $\nu$-balanced embeddings} is the map
\begin{equation}
\begin{split}
\mu^E:\GL(\HH_p)/U(\HH_p)&\longrightarrow\Herm(\HH_p)\\
G&\longmapsto\int_X\,\Pi_{Gz}\,d\nu(z)\,,
\end{split}
\end{equation}
where we identified $X$ with its image in the Grassmanian
$\mathbb{G}(\rk(E),\HH_p)$ under the Kodaira map \eqref{Kodvbdef}.
\end{defi}

Comparing with Definition \ref{nubaldef} of $\TT_{E_p}$
and by Definition \ref{nubaldef} of a $\nu$-balanced Hermitian
product, we have
\begin{equation}\label{muE1=1}
\mu^E(\Id_{\HH_p})=\frac{\Vol(X,\nu)\rk(E)}{\dim\HH_p}\Id_{\HH_p}\,,
\end{equation}
so that the moment map at a balanced embedding is in fact
equal a constant multiple of the identity.
This is the characterisation of a $\nu$-balanced metric
used by Wang in \cite[Th.\,1.1]{Wan05}.

Consider the identification
\begin{equation}\label{GL/U}
\begin{split}
\GL(\HH_p)/U(\HH_p)&\xrightarrow{~\sim~}\Prod(\HH_p)\\
G~&\longmapsto~q_G:=q(G\,\cdot\,,\,G\cdot)\,,
\end{split}
\end{equation}
and recall the natural identification \eqref{Prodid}.
The following result establishes a link between the differential
of Donaldson's map $\TT_{E_p}$ and the moment map $\mu^E$
at a $\nu$-balanced product.

\begin{prop}\label{dmu=1-dT}
Assume that $q\in\Prod(\HH_p)$ is $\nu$-balanced. Then
the differentials of the moment map and Donaldson's map at $q$ satisfy
\begin{equation}
\frac{\dim\HH_p}{\Vol(X,\nu)\rk(E)}D_q\,\mu^E=\Id_{\HH_p}-D_q\TT_{E_p}\,.
\end{equation}
\end{prop}
\begin{proof}
Fix $B\in\End(\HH_p)$, and $z\in\mathbb{G}(\rk(E),\HH_p)$.
Take $s\in\HH_p$ such that $\Pi_zs=s$, so that
$\Pi_{e^{tB}z}\,e^{tB}s=\,e^{tB}s$ for all $t\in\R$, and
differentiating, we get
\begin{equation}\label{dtPi1}
\left(\dt\Big|_{t=0}\Pi_{e^{tB}z}\right)s=\left(\Id_{\HH_p}-\Pi_z\right)Bs\,.
\end{equation}
Take now $s^\perp\in\HH_p$ such that $\Pi_zs^\perp=0$, so that
$\Pi_{e^{tB}z}\,e^{-tB}s^\perp=0$ for all $t\in\R$, and
differentiating, we get
\begin{equation}\label{dtPi2}
\left(\dt\Big|_{t=0}\Pi_{e^{tB}z}\right)s^\perp=\Pi_zBs^\perp\,.
\end{equation}
By equations \eqref{dtPi1} and \eqref{dtPi2}, we thus get that
\begin{equation}\label{dtPi0}
\begin{split}
\dt\Big|_{t=0}\Pi_{e^{tB}z}&=
\left(\dt\Big|_{t=0}\Pi_{e^{tB}z}\right)\Pi_z
+\left(\dt\Big|_{t=0}\Pi_{e^{tB}z}\right)(\Id_{\HH_p}-\Pi_z)\\
&=B\Pi_z+\Pi_z B-2\Pi_z B\Pi_z\,.
\end{split}
\end{equation}
Now consider a Hermitian endomorphism $A\in\Herm(\HH_p)$ as
a tangent vector in $T_q\Prod(\HH_p)$. Via the differential of
\eqref{GL/U}, it is the image of the
endomorphism $B\in\End(\HH_p)$ satisfying $B=A/2$.
Using formulas \eqref{DHTpointfla} and \eqref{evAev=T*A},
one can use Definition \ref{nubaldef} and
Proposition \ref{dFSHilb} to get
\begin{equation}
D_q\TT_{E_p}.\,A=\frac{\dim\HH_p}{\Vol(X,\nu)\rk(E)}
\int_X\Pi_z A\Pi_z\,d\nu(z)\,.
\end{equation}
Thus using identities \eqref{muE1=1} and \eqref{dtPi0}, we get
\begin{equation}
\begin{split}
D_q\,\mu^E.A&
=\frac{1}{2}A\,\mu^E(\Id_{\HH_p})+\frac{1}{2}\mu^E(\Id_{\HH_p})\,A
-\int_X\Pi_z A\Pi_z\,d\nu(z)\\
&=\frac{\Vol(X,\nu)\rk(E)}{\dim\HH_p}\left(A-D_q\TT_{E_p}.\,A\right)\,.
\end{split}
\end{equation}
This completes the proof.
\end{proof}

\begin{rmk}\label{KMSrmk}
Assume that $E$ is Mumford stable and $\nu$ is the Liouville measure.
For any $p\in\N$
big enough, let $q_p\in\Prod(\HH_p)$ be a $\nu$-balanced
Hermitian product furnished by Theorem \ref{balvb}.
Then by Theorem \ref{gapvb} and Propostion
\ref{dmu=1-dT}, for all $A\in\Herm(\HH_p)$,
we have as $p\fl+\infty$,
\begin{equation}\label{KMSth}
\frac{\dim\HH_p}{\Vol(X,\nu)}\Tr^{\HH_p}[A\,D_{q_p}\mu^E(A)]
\geq\left(\frac{\lambda_1^E}{2\pi p}+o(p^{-1})\right)\Tr^{\HH_p}[A^2]\,.
\end{equation}
Via formula \eqref{npfla} for $\dim\HH_p$,
we then recover the result of Keller, Meyer and
Seyyedali in \cite[Th.\,3]{KMS16}. The differential
of the moment map is interpreted in \cite{KMS16} as a quantization
of the Bochner Laplacian. However, its relevance in the study of
Hermite-Einstein metrics is better seen from its interpretation
as the Hessian of the \emph{energy functional}
associated with this moment map problem, as defined for instance
in \cite[(2.3)]{Sey09} under the name of \emph{Donaldson's functional}.
The lower bound \eqref{KMSth} thus gives an estimate on the
convexity of this functional, which plays an instrumental role
in the convergence results of Theorem \ref{TTEcvth} and gives
a natural explanation for
the key lower bound appearing in the work of Wang \cite{Wan05}.
\end{rmk}

Assume now that $E=\C$, fix $p\in\N$ big enough
and let $q\in\Prod(\HH_p)$ be a balanced
Hermitian product in the sense of Definition \ref{baldef}. Recall the identification \eqref{GL/U}. The following definition is a reformulation of the moment map used in \cite[(1)]{Don01}.

\begin{defi}
The \emph{moment map for balanced metrics} is the map
\begin{equation}
\begin{split}
\mu:\GL(\HH_p)/U(\HH_p)&\longrightarrow\Herm(\HH_p)\\
H&\longmapsto\int_X\,\Pi_{Gz}\,\frac{\om_{\FS(q_G)}^n}{n!}(z)\,.
\end{split}
\end{equation}
\end{defi}

As for the moment map for $\nu$-balanced metrics, formula
\eqref{baldef} for $\TT_p$ implies
\begin{equation}
\mu(\Id_{\HH_p})=\frac{\Vol_{\FS(q)}(X)}{\dim\HH_p}\Id_{\HH_p}\,.
\end{equation}
This is the characterisation of a balanced metric
used by Donaldson in \cite{Don01}.

The following result is the analogue of Proposition \ref{dmu=1-dT} for
balanced metrics.

\begin{prop}\label{dmu=1-dT'}
Assume that $q\in\Prod(\HH_p)$ is balanced. Then
we have
\begin{equation}
\frac{\dim\HH_p}{\Vol_{\FS(q)}(X)}D_q\mu=\Id_{\HH_p}-D_q\TT_p\,.
\end{equation}
\end{prop}
\begin{proof}
Fix $A\in\Herm(\HH_p)$, and let $B\in\End(\HH_p)$ be such that
$B=A/2$ as in the proof of Proposition \ref{dmu=1-dT'}.
Then using formulas \eqref{DHTpointfla} and \eqref{evAev=T*A}
as in
the proof of Proposition \ref{dmu=1-dT}, we get from Corollary \ref{dTT},
\begin{equation}
\begin{split}
&D_q\mu.A=\frac{\Vol_{\FS(q)}(X)}{\dim\HH_p}A
-\int_X\Pi_z A\Pi_z\,\frac{\om_{\FS(q_G)}^n}{n!}(z)-
\int_X\,\Pi_z~\dt\Big|_{t=0}\frac{\om_{\FS(q_{e^{tB}})}^n}{n!}(z)\\
&=\frac{\Vol_{\FS(q)}(X)}{\dim\HH_p}\left(A-\int_XT^*_p(A)\,\Pi_z\,\frac{\om_{\FS(q_G)}^n}{n!}-
\frac{1}{4\pi}\int_X\,\Delta_{\FS(q)} T^*_p(A)\,\Pi_z\,
\frac{\om_{\FS(q_G)}^n}{n!}\right)\\
&=\frac{\Vol_{\FS(q)}(X)}{\dim\HH_p}\left(A-D_q\TT_p.A\right)\,,
\end{split}
\end{equation}
where we used formulas \eqref{dFSfla} and \eqref{dtvol} for
the differential of the Fubini-Study volume form.
This completes the proof.
\end{proof}

\begin{rmk}\label{Finermk}
Assume that the assumptions of Theorem \ref{balex} are satisfied,
and for any $p\in\N$
big enough, let $q_p\in\Prod(\HH_p)$ be a balanced
Hermitian product. Then by Theorem \ref{gapvb} and
Proposition \ref{dmu=1-dT'}, for all $A\in\Herm(\HH_p)$,
we have as $p\fl+\infty$,
\begin{equation}\label{Fineth}
\frac{\dim\HH_p}{\Vol_{\FS(q_p)}(X)}
\Tr^{\HH_p}[A\,D_{q_p}\mu(A)]\geq\frac{\mu_1}{8\pi p^2}+o(p^{-2})\,.
\end{equation}
Via formula \eqref{npfla} for $\dim\HH_p$ and the fact that
$\Vol_{\FS(q_p)}=p^n\Vol_h(X)$
for $h^p=\FS(q_p)$, we then recover the result of Fine in
\cite[Th.\,7]{Fin12}, where the differential of the moment map
is interpreted as the Hessian of the associated energy functional
in the same way as in Remark \ref{KMSrmk}. As explained
in \cite[Cor.\,5]{Fin12}, the lower bound
\eqref{Fineth} gives a natural explanation for
the lower bound playing a key role in the works of Donaldson
\cite{Don01} and Phong and Sturm in \cite[Th.\,2]{PS04}.
\end{rmk}

\section{Physical interpretation and examples}
\label{mathphi}

In this Section, we intepret our results on
Berezin-Toeplitz quantization of vector bundles in terms of
\emph{quantum-classical hybrids} in physics, as considered
for instance in \cite{Elz12}.
In Section \ref{SGsec}, we explain how the celebrated
\emph{Stern-Gerlach experiment} can be interpreted a
the fundamental example of a quantum-classical hybrid,
and in Section \ref{SGKodsec}, we illustrate
the significance of Theorem \ref{mainthvb} in this context.
Finally, in Section \ref{qcsec}, we discuss
the
physical interpretation of Theorem \ref{BMSvb} via the
process of quantization in stages described in Section \ref{stagesec}.

\subsection{Stern-Gerlach experiment as a quantum-classical hybrid}
\label{SGsec}

In the physical context of quantum-classical hybrids,
one only quantizes some specified
degrees of freedom, while keeping the others classical.
At the fully classical level, this separation of degrees of freedom
is described by a symplectic fibration, as in Definition \ref{quantfib}.
In this context, the limiting regime when $p$ tends to infinity is 
called the \emph{weak coupling limit} in \cite[\S\,4.5]{GLS96}.
In particular, one expects to
recover the geometry of the quantum-classical hybrid from
the quantization of $(M,\om^M+p\,\om)$
as $p\to+\infty$. In the case when $\pi:(M,\om^M)\to(X,\om)$ is a
fibration of coadjoint orbits, this idea
has already been applied to representation theory by
Guillemin, Sternberg and Lerman in \cite[\S\,4.5]{GLS96}.
Proposition \ref{Tfunct} on the
functoriality of quantization in stages naturally fits into this setting.

To described the basic physical
example of this set-up, recall that any Hermitian vector space
$(V,q)$ can be realized canonically as the Hilbert space of
quantum states associated with the \emph{projectivization} $\IP(V^*)$
of its dual. In fact, the dual of the
\emph{tautological line bundle} of complex lines of $V^*$ over
$\IP(V^*)$ is endowed with the tautological Hermitian metric induced by $q$,
and formula \eqref{preqfib} defines a symplectic form over $\IP(V^*)$.
On the other hand, holomorphic sections of this line bundle
naturally correspond to elements
$v\in V$, and the associated space of quantum states \eqref{quantspace}
is naturally identified with $(V,q)$.
In case $V$ is of complex dimension $2$, this space corresponds to the
space of \emph{spin}-$1/2$, representing
\emph{quantum angular
momentum}, while the projectivization is naturally identified with the sphere
$S^2$, representing \emph{classical angular momentum}. The quantum
number $p\in\N$ thus corresponds to the \emph{spin}, and the $p$-th tensor
power of the prequantizing line bundle gives rise to the Hilbert space of
spin-$p/2$, which is naturally identified with the $p$-th symmetric
$\textup{Sym}^pV$. Finally, the Berezin-Toeplitz quantization of the
Cartesian coordinate functions of $S^2\subset\R^3$ coincide up to a
universal
constant with the usual \emph{spin operators} acting on $\textup{Sym}^p V$,
also called \emph{Pauli matrices} in the spin-$1/2$ case.
This leads to the following fundamental Example of a symplectic fibration.

\begin{ex}\label{exprojectivisation}
Given a holomorphic Hermitian vector bundle
$(E,h^E)$ over a prequantized Kähler manifold $(X,J,\om)$,
the fibration
$\pi:\IP(E^*)\fl X$ obtained by the projectivization of its dual
is prequantized by the dual of the
\emph{tautological line bundle} over $\IP(E^*)$, endowed with the
natural Hermitian metric induced by $h^E$.
Then the associated quantum-classical hybrid of Definition \ref{qchybdef} 
naturally coincides with $(E,h^E)$ over $(X,\om)$.
The particular case of a
rank-$2$ vector bundle gives a fibration of spheres over $X$,
reproducing a situation of Stern-Gerlach type.
\end{ex}

\subsection{Quantum noise in the Stern-Gerlach experiment}
\label{SGKodsec}

From the point of view of quantum-classical hybrids,
Theorem \ref{mainthvb} says that the quantum noise of its
Berezin-Toeplitz quantization is controled by the spectrum of the associated
Kodaira Laplacian \eqref{KodLapdef}. In this Section, we describe this
spectrum
in various cases in the physical context of the Stern-Gerlach experiment
introduced in Example \ref{exprojectivisation}.

In the case when $E=\C$, the Kodaira Laplacian \eqref{KodLapdef}
coincides with the Laplace-Beltrami operator $\Delta$ of $(X,g^{TX})$ acting
on $\cinf(X,\C)$, and for general $(E,h^E)$, the Laplacian
\eqref{KodLapdef} can be obtained as the operator
induced by the \emph{horizontal Laplace-Beltrami operator} of
$\pi:\IP(E^*)\to X$, seen as
a Kähler fibration in the sense of \cite[Def.\,1.4]{BGS88}.
In many physical situations involving quantum-classical hybrids,
the relevant observables
$F\in\cinf(X,\End(E))$ under consideration
can be diagonalized in a direct sum
$E=\oplus_{j=1}^m E_j$ of
holomorphic Hermitian subbundles of $(E,h^E)$. Writing
$F=\sum_{j=1}^m\,f_j\,P_j$, with $f_j\in\cinf(X,\C)$ and where
$P_j\in\End(E)$ are the orthogonal projectors on $E_j$,
the Laplacian \eqref{KodLapdef} is then given
by the formula
\begin{equation}\label{Boxdiag}
\Box\,F=\sum_{j=1}^m\,\left(\Delta f_j\right)\,P_j\,.
\end{equation}
On the other hand, in the fundamental case of the Stern-Gerlach experiment,
one considers $E=\C^2$ the trivial rank-$2$ vector
bundle, which describes classical particles over $X$ of quantized
spin-$\frac{1}{2}$. Then all
$F\in\cinf(X,\End(\C^2))$ can be decomposed as
$F=\sum_{j=1}^3\,f_j\,\sigma_j$,
with $f_j\in\cinf(X,\C)$ and where
$\sigma_j\in\End(\C^2)$ are the standard \emph{Pauli matrices},
for all $1\leq j\leq 3$. In that case,
the Laplacian \eqref{KodLapdef} is again given
by the formula
\begin{equation}
\Box\,F=\sum_{j=1}^3\,\left(\Delta f_j\right)\,\sigma_j\,.
\end{equation}
This gives a natural interpretation to the Kodaira Laplacian, and thus to
the lower bound \eqref{eq-LB} on the quantum noise induced by the Berezin-Toeplitz quantization of the Stern-Gerlach experiment.

\begin{ex}\label{Kodex}
To illustrate the relevance of Theorem \ref{mainthvb} in a concrete
situation,
we will compute the spectrum of the Laplacian
\eqref{KodLapdef}
in the case of the holomorphic Hermitian vector bundle
$E_{(k)}:=\C\oplus L^k$ over $X=\CP^1$
for $k\in\N$,
where $L$ is the dual of the tautological line bundle equipped with
the natural Fubini-Study Hermitian metric.
Following \cite{Koh93}, let us first describe the
Laplacian $\Box_k$ acting on the space of sections of the holomorphic line bundle $L^k$ over $\C P^1$, defined by \eqref{KodLapdef} with $L^k$
instead of $E$.
Set $G=SU(2)$, and write $S^1$ for the maximal torus of $G$ so that
$\C P^1 = G/S^1$. The total space of the line bundle $L^k$ is given by
$(G \times \C)/S^1$, where $S^1$ acts
on $\C$ via the representation $\phi \mapsto e^{-ik\phi}$. The sections
of $L^k$ then identify with the
functions $f: G \to \C$ equivariant under the action of $S^1$. 
By the Peter-Weyl theorem, $L_2(G)$ splits as $\oplus (m+1)V_m$, where $m \geq 0$ and
$V_m$ denotes thet $(m+1)$ dimensional irreducible unitary representation
of $G$. Write $e_{j,m}$ for a vector of height $j\in[-m,m]$ in $V_m$, with
$j-m\in 2\N$. With this notation and the identification above, we get
\begin{equation}
L_2(\CP^1, L^k) = \bigoplus_{m\in 2\N+|k|} (m+1)\Span(e_{k,m})\;.
\end{equation}
Denote by $C$ the \emph{Casimir operator} of $G$ on $L_2(G)$, which
acts as the scalar operator $-m(m+2) \Id$ on $V_m$, for all
$m \geq 0$. By the general version of the Weitzenb\"ock formula of
Proposition \ref{Weitzenfla}, the Laplacian
\eqref{Kodex} on $L_2(\CP^1, L^k)$ is given by
\begin{equation}
\label{eq-KL-1}
\Box_k = - \frac{1}{2}C - \frac{k(k+2)}{2} \Id\;.
\end{equation}
Setting $m = |k| + p$ for $p\in 2\N$,
it follows that the spectrum $S_k$ of $\Box_k$ has the following form:

\medskip\noindent For $k=0$, the spectrum $S_0$ consists of eigenvalues $p(p+2)/2$ with multiplicity $p+1$, where $p\in 2\N$.

\medskip\noindent For $k > 0$, the spectrum $S_k$ consists of eigenvalues
$\left(2pk +p(p+2)\right)/2$ with multiplicity $k+p+1$, where $p\in 2\N$.

\medskip\noindent For $k < 0$, the spectrum $S_k$ consists of eigenvalues 
$\left( (2p+4)|k| + p(p+2)\right)/2$ with multiplicity $|k|+p+1$, where $p\in 2\N$.

\medskip

We apply this consideration to the quantum-classical hybrid
given by the Hirzebruch surface
$M_k:= \IP(E_{(k)})$ over $X=\CP^1$. It follows that
$$\End (E_{(k)}) = \C \oplus \C \oplus L^k \oplus L^{-k}\;,$$
and the Laplacian \eqref{KodLapdef} splits as
$\Box = \Box_0 \oplus \Box_0 \oplus \Box_k \oplus \Box_{-k}$. 
The corresponding spectrum with multiplicities is
$$\Sigma_k = 2S_0 \cup S_k \cup S_{-k}\;.$$
In particular, the spectra are different for different $k\in\N$: this can be seen by looking at the minimal positive eigenvalue of $S_k \cup S_{-k}$ which, for $k > 1$ equals
$2k$. Thus, the spectrum of the Kodaira Laplacian on $\End(E_{(k)})$ 
determines the number $k\in\N$. The latter is a holomorphic
invariant of the Hirzebruch 
surface $M_k$, for all $k\in\N$. Namely, the only strictly negative self-intersection of an irreducible
curve on $M_k$ equals $-k$, see \cite[p.141]{BPV84}. It would be interesting to understand
which biholomorphic invariants of more sophisticated complex manifolds "can be heard" with
the Kodaira Laplacian.
\end{ex}  

\subsection{Quantum-classical correspondence for quantum-classical hybrids}
\label{qcsec}

Theorem \ref{BMSvb} has a natural
interpretation in the context of quantization in stages, where one
considers the quantum-classical
hybrid associated with a prequantized fibration
$\pi:(M,\om^M)\to (X,\om)$ as in Definition \ref{preqfib}
at the weak coupling limit, when $p\to+\infty$.
To explain this point, let us consider the setting of
Proposition \ref{Tfunct}, where
$T_{\pi}:\cinf(M,\R)\to\cinf(X,\Herm(E))$
denotes the Berezin-Toeplitz
quantization of the fibres of $\pi:(M,\om^M)\to (X,\om)$,
and for any $f,\,g\in\cinf(M,\R)$, set $F:=T_\pi(f)$ and $G:=T_\pi(g)$
in Theorem \ref{BMSvb}.
For any $p\in\N$
big enough, write $\{\cdot,\cdot\}$ for the vertical Poisson bracket 
induced by $\om^M$ in the fibres of $\pi:M\to X$,
write
$\{\cdot,\cdot\}_p$ for the Poisson bracket over $\cinf(M,\R)$
induced by the symplectic form $\om^M+p\pi^*\om$,
and write $g_p^{TX}$ for the associated Kähler metric on $M$,
inducing a Hermitian product $\<\cdot,\cdot\>_{g_p^{TX}}$ on $T^*M$.
Then at the weak
coupling limit $p\to+\infty$, we get by definition
\begin{equation}\label{Poissonexp}
\begin{split}
\{f,g\}_p&=\sqrt{-1}\left(\<\partial f,\dbar g\>_{g_p^{TM}}-\<\partial f,\dbar g\>_{g_p^{TM}}\right)\\
&=\{f,g\}+\sqrt{-1}\left(\<\partial^H f,\dbar^H g\>_{g_p^{TM}}
-\<\partial^H f,\dbar^H g\>_{g_p^{TM}}\right)\\
&=\{f,g\}
+\frac{\sqrt{-1}}{p}\left(\<\partial^H f,\dbar^H g\>_{\om}-\<\partial^H f,\dbar^H g\>_{\om}\right)+O(p^{-2})\\
&=\{f,g\}
+\frac{1}{p}\,\pi^*\om(\xi_f^H,\xi_g^H)+O(p^{-2})\,,
\end{split}
\end{equation}
where $d^Hf=\partial^Hf+\dbar^Hf$ denotes the restriction to the horizontal
tangent space $T^HM\subset TM$ of the fibration,
where $\<\cdot,\cdot\>_{\om}$ denotes the Hermitian metric on $T^{H,*}M$
induced by $\pi^*\om$ and
where $\xi_f^H\in\cinf(M,T^HM)$ is defined as the unique horizontal
vector field satisfying $d^Hf=\pi^*\om(\cdot,\xi_f^H)$.
We now claim that
the first and second order coefficients
in the expansion \eqref{BMSvbfla} can be respectively interpreted as
the quantization of the first and second order coefficients
in the expansion \eqref{Poissonexp}. To see this, replace $\om^M$
by $r \om^M$ for some $r\in\N$ in Definition \ref{preqfib},
so that the expansion \eqref{Poissonexp} holds with $\{\cdot,\cdot\}$
replaced by $r^{-1}\{\cdot,\cdot\}$.
Then Proposition \ref{Tfunct} and Theorem \ref{BMSvb}, together with
formula \eqref{C1}, imply that as
$r\to+\infty$, we get
\begin{equation}\label{rexp}
\begin{split}
T_{E_p}([F,G])&=\frac{\sqrt{-1}}{2\pi r}T_{E_p}(T_\pi(\{f,g\})+O(r^{-2})=\frac{\sqrt{-1}}{2\pi r}T_p(\{f,g\})+O(r^{-2})\,,\\
T_{E_p}(C(F,G))&=\sqrt{-1}T_{E_p}
\left(\<T_\pi(\partial^H f),T_\pi(\dbar^H g)\>_{\om}-\<T_\pi(\partial^H f),T_\pi(\dbar^H g)\>_{\om}\right)+O(r^{-1})\\
&=\sqrt{-1}T_{E_p}\left(T_\pi\left(\<\partial^H f,\dbar^H g\>_{\om}-\<\partial^H f,\dbar^H g\>_{\om}\right)\right)+O(r^{-1})\\
&=T_p(\pi^*\om(\xi_f^H,\xi_g^H))+O(r^{-1})\,,
\end{split}
\end{equation}
where we used the fact that
$\nabla^{\End E}T_\pi(f)=T_\pi(d^Hf)+O(r^{-1})$
due to Ma and Zhang in \cite[Th.\,0.8]{MZ07}.
Theorem \ref{BMSvb} thus states that the Lie bracket
$[T_p(f),T_p(g)]=[T_{E_p}(F),T_{E_p}(G)]$ is a quantization of
the Poisson bracket $\{\cdot,\cdot\}_p$ at the weak coupling limit
when $p\to+\infty$. Note that this interpretation cannot
be obtained as a consequence of Theorem \ref{BMSvb} for $E=\C$
over $(M,\om^M+p\,\pi^*\om)$, since
the $2$-form $\pi^*\om$ is degenerate along the fibres of $\pi:M\fl X$,
so that the limiting regime of $(M,\om^M+p\,\pi^*\om)$ as $p\to+\infty$
cannot be reduced to the usual semi-classical limit over $(M,p\,\til{\om})$
as $p\to+\infty$ for some symplectic form $\til{\om}\in\Om^2(M,\R)$.

Theorem \ref{BMSvb} also shows that the Berezin-Toeplitz quantization
of $(E,h^E)$ over $(X,p\,\om)$ as $p\to+\infty$ provides a deformation
of the pointwise Lie bracket $[\cdot,\cdot]$ over $\CC^\infty(X,\Herm(E))$,
so that the coefficient \eqref{C1}
forms a degree-$1$ cocycle for this Lie algebra, which means 
that for all $F,\,G,\,H\in\cinf(X,\Herm(E))$, we have
\begin{multline}\label{cocycle}
[C(F,G),H]+[C(H,F),G]+[C(G,H),F]\\
+C([F,G],H)+C([H,F],G)+C([G,H],F)=0\,.
\end{multline}
In case $F,\,G\in\cinf(X,\Herm(E))$ are scalar endomorphisms in each fibre,
equation \eqref{C1} shows in addition
that $C(F,G)$ coincides with
the Poisson bracket on $\cinf(X,\R)$ induced by $\om$.
This is the quantum counterpart of the expansion \eqref{Poissonexp},
which provides a deformation of the
vertical Poisson bracket $\{\cdot,\cdot\}$ over $\CC^\infty(M,\R)$.
Specifically, the expansion \eqref{Poissonexp} shows that
the coefficient $c(f,g):=\pi^*\om(\xi_f^H,\xi_g^H)$, for all
$f,\,g\in\cinf(M,\R)$,
forms a degree-$1$ cocycle for this Poisson algebra, and coincides
with the
Poisson bracket induced by $\om$ on the base of $\pi:M\to X$
when $f,\,g\in\cinf(M,\R)$ are constant in the fibres.
It would be interesting to figure out the algebraic
properties of such cocycles, and deduce from them invariants of the
mixed dynamics of a quantum-classical hybrid. In the following Example,
we describe a physically relevant situation where such mixed dynamics
can be fully described.

\begin{ex}\label{SGqc}
\textbf{(Generalized Stern-Gerlach)}
Consider the physically relevant situation described in Section
\ref{SGKodsec}, where one considers
a subalgebra of observables which are diagonal in a direct sum
$E=\oplus_{j=1}^m E_j$ of
holomorphic Hermitian subbundles of $(E,h^E)$.
Then the cocycle \eqref{C1} preserves this subalgebra.
Specifically, consider $F,\,G\in\cinf(X,\Herm(E))$ such that
$F=\sum_{j=1}^m\,f_j\,P_j$
and $G=\sum_{j=1}^m\,g_j\,P_j$, for some $f_j,\,g_j\in\cinf(X,\C)$ and
$P_j\in\End(E)$ the orthogonal projector on $E_j$, for 
each $1\leq j \leq m$. Then equation \eqref{C1} gives
\begin{equation}\label{SGC}
C(F,G)=\sum_{j=1}^m\,\{f_j,g_j\}\,P_j\,,
\end{equation}
where $\{\cdot,\cdot\}$ denotes the usual Poisson bracket induced
by $\om$ on $\cinf(X,\R)$. In particular, the time evolution
of the observable $F$ under the Hamiltonian flow generated by $G$
over the quantum-classical hybrid is described by the ordinary differential
equations
\begin{equation}
\dot{f_j} = \{g_j,f_j\}\,,\quad 1\leq j \leq m\;.
\end{equation}
Theorem \ref{BMSvb} thus
shows that the dynamics on the quantum-classical
hybrid described by this subalgebra of observables
reduces to the classical dynamics induced by a collection of 
Hamiltonians related to
each different quantum state, represented by the corresponding eigenspace.
This is precisely the case for the usual Stern-Gerlach experiment,
in which case $E=X\times\C^2$ and
a classical particle goes up or down depending on either
its quantum spin is $(1,0)$ or $(0,1)\in\C^2$.
\end{ex}

\begin{rmk}\label{otherlimit}
Let us write $(E_{(r)},h^{E_{(r)}})$
for the quantum-classical hybrid associated with
$\pi:(M,r\,\om^M)\to (X,\om)$, for every $r\in\N$.
For the first line of \eqref{rexp}, we used the fact that the
Berezin-Toeplitz
quantization $T_{\pi,r}:\cinf(M,\R)\to\cinf(X,\Herm(E_{(r)}))$ of the fibres
of $\pi:(M,r\,\om^M)\to (X,\om)$ satisfies the quantum-classical
correspondence of Theorem \ref{BMSvb} for $E=\C$
in each fibre as $r\to+\infty$.
For the second line
of \eqref{rexp}, we used furthermore a technical
result of Ma and Zhang in \cite[Th.\,0.4]{MZ07}, where it is also shown that
the trace-free part $R^{E_{(r)}}_0\in\Om^2(X,\End E_{(r)})$
of the Chern curvature
of $(E_{(r)},h^{E_{(r)}})$
satisfies the following asymptotic expansion in the operator norm
as $r\to+\infty$,
\begin{equation}
\frac{1}{r}R^{E_{(r)}}_0=T_{\pi,r}(\rho)+O(r^{-1})\,,
\end{equation}
where $\rho\in\cinf(M,\Lambda^2T^HM^*)$ is the \emph{symplectic curvature}
of the Hamiltonian fibration $\pi:(M,\om^M)\to (X,\om)$, as described
for instance in \cite{Pol96}.
This answers a question of Savelyev and Shelukhin in
\cite[Rmk.\,7.2]{SS20}.
\end{rmk}



\providecommand{\bysame}{\leavevmode\hbox to3em{\hrulefill}\thinspace}
\providecommand{\MR}{\relax\ifhmode\unskip\space\fi MR }
\providecommand{\MRhref}[2]{%
  \href{http://www.ams.org/mathscinet-getitem?mr=#1}{#2}
}
\providecommand{\href}[2]{#2}

\Addresses

\end{document}